\documentclass{article}[10pt]
\usepackage{theorem}
\usepackage{eurosym}
\usepackage{amssymb}
\usepackage{float}
\usepackage{amsmath}
\usepackage{amsfonts}
\usepackage[american]{babel}
\usepackage{verbatim}
\usepackage{geometry}
\usepackage{color}

\newtheorem{theorem}{Theorem}

\newtheorem{corollary}[theorem]{Corollary}

\newtheorem{definition}[theorem]{Definition}

{\theorembodyfont{\rmfamily}

}

\newtheorem{lemma}[theorem]{Lemma}

\newtheorem{proposition}[theorem]{Proposition}
{\theorembodyfont{\rmfamily}
\newtheorem{remark}[theorem]{Remark}
}

\newenvironment{proof}[1][Proof]{\noindent\textbf{#1} }{\ \rule{0.5em}{0.5em}}

\newcommand*\re{\mathbb{R}}

\newcommand*\Omegabar{\overline{\Omega}}
\newcommand*\delomega{\partial\Omega}
\newcommand*\cnn{\omega_{n-1}^{1/(n-1)}}
\newcommand*\intdelomega{\int_{\partial\Omega}}
\newcommand*\intomega{\int_{\Omega}}

\newcommand*\WBOmega{\mathcal{B}_1(\Omega)}

\newcommand*\WBDiskrad{W^{1,n}_{0,rad}(B_1)\cap \mathcal{B}_1(B_1)}

\newcommand*\dxb{{|x|^{\beta}}}
\newcommand*\dH{d\mathcal{H}^{n-1}}

\begin{document}


\title{Extremals for the Singular Moser-Trudinger Inequality via n-Harmonic Transplantation}
\maketitle
\date{}

\centerline{\scshape Gyula Csat\'{o}$^1,$ Prosenjit Roy$^2$ and Van Hoang Nguyen$^3$ }
\medskip 
{\footnotesize
 \centerline{1. Universitat Polit\`ecnica de Catalunya and member of BGSMath Barcelona, Spain, }
 \centerline{supported by Fondecyt grant no. 11150017, by the Mar\'ia de Maeztu Grant MDM-2014-0445 }
  \centerline{ and by MINECO grant MTM2017-84214-C2-1-P.} 
   \centerline{2. Indian Institute of Teechnology, Kanpur, India}
   \centerline{3. Institute of Mathematics, Vietnam Academy of Science and Technology,}
   \centerline{Hanoi, Vietnam}

}

\smallskip

\begin{abstract}
The Moser-Trudinger embedding has been generalized in [Adimurthi A.; Sandeep K., A singular Moser-Trudinger embedding and its applications, \textit{NoDEA Nonlinear Differential Equations Appl.}, 13 (2007), no. 5-6, 585--603] to the following weighted version: if  $\Omega\subset\mathbb{R}^n$ is bounded, $\omega_{n-1}$ is the $\mathcal{H}^{n-1}$ measure of the unit sphere, then for  $\alpha>0$ and $\beta\in [0,n)$,
$$
  \sup_{u\in\mathcal{B}_1}\int_{\Omega}\frac{e^{\alpha |u|^{n/(n-1)}}}{|x|^{\beta}}\leq C \  \Leftrightarrow  \   \frac{\alpha}{\alpha_n}+\frac{\beta}{n}\leq1,\qquad  
$$ where 
$\alpha_n=n\cnn$ and $\mathcal{B}_1 =  \left\{ u \in W_0^{1, n}(\Omega) \  |  \ \int_{\Omega} |\nabla u |^n \leq1 \right\}$.
We prove that the supremum is attained  on any domain $\Omega$. The paper also fills in the gaps in the proof of [Lin K.C.,  Extremal functions for Moser's inequality, \textit{Trans. of. Am. Math. Soc.}, 384 (1996), 2663--2671], which deals with the case $\beta=0.$ \end{abstract}

\let\thefootnote\relax\footnotetext{\textit{2010 Mathematics Subject Classification.} Primary 26D10 (46E35, 52A40) }
\let\thefootnote\relax\footnotetext{\textit{Key words and phrases.} Moser Trudinger inequality, extremal function.}

\section{Introduction}\label{section:introduction}

Let $\Omega\subset\re^n$ be a bounded open smooth set.
The Moser-Trudinger imbedding, which is due to Trudinger \cite{Trudinger} and in its sharp form to Moser \cite{Moser}, states that the following supremum is finite
$$
  \sup_{\substack {v\in W^{1,n}_0(\Omega) \\ \|\nabla v\|_{L^n}\leq 1}}
  \int_{\Omega}\left(e^{n\cnn v^\frac{n}{n-1}}-1\right)<\infty.
$$
First it has been shown by Carleson-Chang \cite{Carleson-Chang} that the supremum is actually attained, if $\Omega$ is a ball. In \cite{Struwe 1}, Struwe proved for $n=2$ that the result remains true if $\Omega$ is close to a ball. Using the harmonic transplantation method Flucher \cite{Flucher} generalized this result to arbitrary  bounded domains in $\re^2.$ If $n=2$ and $\Omega$ is the unit disk some new proofs have been obtained of the Moser-Trudinger inequality, respectively of the Carleson-Chang's result in Mancini-Martinazzi \cite{Gabriele Mancini and Martinazzi}. Other results dealing with critical points and extremal functions of the Moser-Trudinger inequality, or its sharper version, have also been obtained in  Malchiodi-Martinazzi \cite{Malchiodi-Martinazzi} (for $\Omega$ equal the unit disk), Adimurthi-Druet \cite{Adimurthi-Druet} and Yang \cite{Yunyan Yang 2006}, \cite{Yunyan Yang 2006 manifold} (for a version on manifolds) and \cite{Yunyan Yang 2015 JDE}. These results use blow up analysis and 
are all 
restricted to $2$ dimensions.
See also Adimurthi-Tintarev \cite{Adi-Tintarev}, Mancini-Sandeep \cite{manchini}, \cite{ngo}, \cite{ngyuen}, \cite{taka} and Yang \cite{Yunyan Yang manifold 2} and the references in these papers for other recent developments on the subject. 
However, the only result dealing with the extremal functions in higher dimensions $n\geq 3$
and more general bounded domains than balls, is by Lin [25]. His proof is however rather sketchy
in some parts and we believe he missed several important and nontrivial points. In particular the
entire construction of Section \ref{section:domain to ball} in this paper is necessary in the case $\beta=0$ too, as well as the use
of all the nontrivial results on the $n$-Laplacian we summarized in the Appendix. Setting $\beta=0$ in
this present paper we implicitly give rigorous proofs to all of Lin’s claims.

The Moser-Trudinger embedding has been generalized by Adimurthi-Sandeep \cite{Adi-Sandeep} to a singular version, which reads as the following:
If $\alpha>0$ and $\beta\in[0,n)$ is such that
\begin{equation}
 \label{intro:eq:alpha and beta sum}
  \frac{\alpha}{\alpha_n}+\frac{\beta}{n}\leq 1,\quad\text{ where }\alpha_n= n\cnn,
\end{equation}
$\omega_{n-1}$ is the $\mathcal{H}^{n-1}$ measure of the unit sphere,
then the following supremum is finite
\begin{equation}
 \label{eq:intro:supremum sing moser trud}
  \sup_{\substack {v\in W^{1,n}_0(\Omega) \\ \|\nabla v\|_{L^n}\leq 1}}
  \int_{\Omega}\frac{e^{\alpha |v|^\frac{n}{n-1}}-1}{|x|^{\beta}}<\infty\,.
\end{equation}
In the case $n=2$ it was proven first in Csat\'o-Roy \cite{Csato Roy Calc var Pde}, \cite{Csato Roy Comm in PDE} that the supremum is attained. Afterwards, using blow-up analyis, the same result and some refinements have also been obtained by Yang and Zhu \cite{Yunyan Yang and Zhu JFA 2017}. Their proof is based on classification theorems in $2$ dimensions by Chen and Li \cite{Chen Li Duke 1991}, \cite{Chen Li Duke 1995}. Recently, also in $2$ dimensions, a third proof was given by Lula-Mancini \cite{Lula and Mancini Gabriele}, also depending on \cite{Chen Li Duke 1991}.  We point out that there is no available Chen-Li type classification result in higher dimensions. The aim of this paper is to generalize \cite{Csato Roy Calc var Pde} to higher dimensions and we prove the following theorem.

\begin{theorem}\label{theorem:intro:Extremal for Singular Moser-Trudinger}
Let $n\geq 2$ and $\Omega\subset\re^n$ be a bounded open smooth set. Let $\alpha>0$ and $\beta\in[0,n)$ be such that \eqref{intro:eq:alpha and beta sum} is satisfied.
Then there exists $u\in W_0^{1,n}\left(\Omega\right)$ such that $\|\nabla u\|_{L^n}\leq 1, $  $u\geq 0$ and
$$
  \sup_{\substack {v\in W^{1,n}_0(\Omega) \\ \|\nabla v\|_{L^n}\leq 1}}
  \int_{\Omega}\frac{e^{\alpha |v|^{\frac{n}{n-1}}}-1}{|y|^{\beta}}dy=\int_{\Omega
  }\frac{e^{\alpha u^{\frac{n}{n-1}}}-1}{|y|^{\beta}}dy.
$$
\end{theorem}

 It is an immediate consequence of this theorem, that $u\in W_0^{1,n}$ gives a weak solution, for some $\mu>0,$ to the following semilinear elliptic eqaution involving the $n$-Laplacian
$$
   \left\{\begin{array}{rl}
   -\Delta_n u=\mu\,u^{\frac{1}{n-1}}\, e^{\alpha u^{\frac{n}{n-1}}}\,\frac{1}{|y|^{\beta}} & \text{ in
   }\Omega
   \smallskip \\
   u\geq 0 & \text{ in }\Omega
   \smallskip \\
   u=0 & \text{ on }\partial\Omega,
   \end{array}\right.
$$
satisfying moreover $\int_{\Omega}|\nabla u|^n=1.$

The proof of Theorem \ref{theorem:intro:Extremal for Singular Moser-Trudinger} follows the ideas of Flucher \cite{Flucher}, respectively Lin \cite{Lin K C} and we have given an overview of the method in Csat\'o-Roy \cite{Csato Roy Calc var Pde}, which we will not repeat here.
The case $0\notin\Omegabar$ is rather elementery and exactly as in the $2$-dimensional case, see \cite{Csato Roy Calc var Pde}, and the main difficulty is the case $0\in\Omega.$
The proof of Theorem \ref{theorem:intro:Extremal for Singular Moser-Trudinger}, roughly speaking, consists in reducing the problem to a kind of isoperimetric problem with density involving the $n$-Green's function. The $n$-Green's function $G_{\Omega,0}$ of a general domain $\Omega$ containing the origin has the form
$$
  G_{\Omega,0}(y)=-\frac{1}{\cnn}\log|y|-H(y),
$$
where $H$ is the regular part. The conformal incenter at $0$ is then defined by
$$
  I_{\Omega}(0)=e^{-\cnn H(0)}.
$$
This is usually called the conformal inradius at $0,$ see Flucher \cite{Flucher book} and we have adopted the name conformal incenter at $0$ in \cite{Csato Roy Calc var Pde} mistakenly. However, in this paper, we will stick to this name for consistency with \cite{Csato Roy Calc var Pde}, \cite{Csato Roy Comm in PDE}. 
The isoperimetric problem to which the question of the existence of extremal function is reduced is the following inequality:
\begin{equation}
 \label{eq:intro:main conjecture for G}
    \omega_{n-1}^{\frac{n}{n-1}}I_{\Omega}^{n-\beta}(0)
  \leq
  \int_{\partial\Omega}\frac{1}{|y|^{\beta}|\nabla G_{\Omega,0}(y)|}\dH(y)\quad
  \text{ for any $\Omega$ with $0\in\Omega$}.
\end{equation}
There is equality for balls cetered at the origin. Let us now point out the two main differences and difficulties compared to the $2$-dimensional case:
\smallskip

\textit{1.} The reduction of the problem to \eqref{eq:intro:main conjecture for G} uses the $n$-harmonic transplantation. This uses existence, regularity and other properties of th $n$-Laplace equation, which is a linear equation only if $n=2.$ Solutions to the $n$-Laplace equation have moreover weaker regularity properties if $n> 2.$ The difficulties are more of technical type and are mostly relevant in Section \ref{section:domain to ball}.
\smallskip

\textit{2.} 
If $n=2$ then \eqref{eq:intro:main conjecture for G} was proven by the following three inequalities (where $R_{\Omega}$ is the radius of $\Omega^{\ast},$ i.e. $\pi R_{\Omega}^2=|\Omega|$)
\begin{align*}
 \omega_1^2I_{\Omega}^{2-\beta}(0)
 \,\substack{\phantom{\text{(a)}}\\ \leq \\ \text{(a)}}&\, 
 \omega_1^2 R_{\Omega}^{2-\beta}=\omega_1^2\left(\frac{|\Omega|}{\pi}\right)^ 
 {\frac{2-\beta}{2}}
  \,\substack{\phantom{\text{(b)}}\\ \leq \\ \text{(b)}}\, 
  \left(\intdelomega\frac{1}{|y|^{\beta/2}}d\mathcal{H}^1(y)\right)^2
  \smallskip \\
  =& \left(\intdelomega\frac{\sqrt{|\nabla G|}}{\sqrt{|y|^{\beta}|\nabla G|}}d\mathcal{H}^1(y)\right)^2
  \,\substack{\phantom{\text{(c)}}\\ \leq \\ \text{(c)}}\,
  \intdelomega\frac{1}{|y|^{\beta}|\nabla G|}d\mathcal{H}^1(y).
\end{align*}
The estimate (a), i.e. $I_{\Omega}\leq R_{\Omega}$ is standard, (b) is a weighted isoperimetric inequality and (c) is just H\"older inequality and the elementery property $\int_{\delomega}|\nabla G|=1.$ This is the method followed by Flucher \cite{Flucher}, Csat\'o-Roy \cite{Csato Roy Calc var Pde}. If $\beta=0,$ then the same steps work in any dimension, just by using the classical isoperimetric inequality and this is what was used by Lin \cite{Lin K C}. If $\beta\neq0,$ then the corresponding weighted isoperimetric inequality was proven by Csat\'o \cite{Csato DIE} if $n=2.$ However, a higher dimensional version of the appropriate weighted isoperimetric inequality fails, see Csat\'o \cite{Csato Isop with density n dim} Theorem 9 (ii) and Example 11. Example 11 shows precisely that such an isoperimetric inequality even fails for any ball not centered at the origin. Moreover, if $n\geq 3$ both estimates (a) and (c) are too generous to prove 
\eqref{eq:intro:main 
conjecture 
for G}, as numerical evidence suggests. So a completely new method had to be developed to prove the higher dimensional case, without using any of the estimates (a)-(c).This new method relies on a careful analysis of the properties of $G_{\Omega,0}$ and on a different weighted isoperimetric inequality by Alvino, Brock, Chiacchio, Mercaldo and Posteraro \cite{Alvino} Theorem $1.1$ (ii), which reads as.
$$
   \int_{\Omega}|y|^{-\beta} dy \leq C\left(\int_{\partial \Omega} |y|^{-\frac{n-1}n\beta} d\mathcal{H}^{n-1}(y)\right)^{\frac{n}{n-1}}
   \quad\text{ for any smooth set $\Omega\subset\re^n,$}
$$
where $C$ is such that equality holds for balls centered at the origin.
The proof of \eqref{eq:intro:main conjecture for G} is contained in Section \ref{section:n greens function}.
\smallskip

\section{Notations and Definitions}\label{section:notation}

Throughout this paper $n\geq 2$ is an integer and $\Omega\subset\re^n$ will denote a bounded open set with smooth boundary $\delomega.$ Its $n$-dimensional Lebesgue measure is written as $|\Omega|.$ The $(n-1)$-dimensional Hausdorff measure is denoted by $\mathcal{H}^{n-1}.$ Balls with radius $R$ and center at $x$ are written $B_R(x)\subset\re^n;$ if $x=0,$ we simply write $B_R$. The space $W^{1,n}(\Omega)$ denotes the usual Sobolev space of functions and $W^{1,n}_0(\Omega)$ those Sobolev functions with vanishing trace on the boundary. Throughout this paper $\alpha,\beta\in \re$ are two constants satisfying $\alpha>0,$ $\beta\in[0,n)$ and 
$$
  \frac{\alpha}{\alpha_n}+\frac{\beta}{n}\leq 1,\quad\text{ where }\alpha_n= n\cnn,
$$
where $\omega_{n-1}$ is the $\mathcal{H}^{n-1}$ measure of the unit sphere.

$-$ We define the the funtctionals $F_{\Omega},J_{\Omega}:W_0^{1,n}(\Omega)\to \re$ by
\begin{align*}
 F_{\Omega}(u)=&\int_{\Omega}\frac{e^{\alpha |u|^{\frac{n}{n-1}}}-1}{|x|^{\beta}}\,dx, \smallskip \\
 J_{\Omega}(u)=&\int_{\Omega}\left(e^{\alpha_n |u|^{\frac{n}{n-1}}}-1\right)dx. \smallskip \\
\end{align*}

$-$ We say that a sequence $\{u_i\}\subset W_0^{1,n}(\Omega)$ concentrates at $x\in\Omegabar$ if 
$$
  \lim_{i\to\infty}\|\nabla u_i\|_{L^n}=1\quad\text{ and }\quad\forall\;\epsilon>0\quad
  \lim_{i\to\infty}\int_{\Omega\backslash B_{\epsilon}(x)}|\nabla u_i|^n=0.
$$
This definition is equivalent to the convergence $|\nabla u_i|^ndx\rightharpoonup \delta_x$ weakly in measure, where $\delta_x$ is the Dirac measure at $x.$   We will use the following well known property of concentrating sequences: if $\{u_i\}$ concentrates, then $u_i\rightharpoonup 0$ in $W^{1,n}(\Omega).$ In particular
\begin{equation} 
 \label{eq:properties of concentrating sequences}
  u_i\to 0\quad\text{ in }L^q(\Omega)\quad\text{ for all }q<\infty.
\end{equation}
\smallskip

$-$ We define the sets
$$
  W^{1,n}_{0,rad}(B_1)=\left\{u\in W^{1,n}_0(B_1)\,\Big|\, u\text{ is radial }\right\}
$$
and analogously $C_{c,rad}^{\infty}(B_1)$ is the set of radially symmetric smooth functions with compact support in $B_1$.
By abuse of notation we will usually write $u(x)=u(|x|)$ for $u\in W^{1,n}_{0,rad}(B_1).$
The space $C_{c,rad}^{\infty}(B_1)$ is dense in 
$W^{1,n}_{0,rad}(B_1)$ in the $W^{1,n}$ norm. If in addition $u$ is radially decreasing we write $u\in W^{1,n}_{0,rad\searrow}(B_1),$ respectively $u\in C_{c,rad\searrow}^{\infty}(B_1).$
\smallskip

$-$ Define 
$$
  \mathcal{B}_1(\Omega)=\left\{u\in W^{1,n}_0(\Omega)\,\big|\,\|\nabla u\|_{L^n}\leq 1\right\}.
$$

\smallskip 

$-$ Finally we define
$$
F_{\Omega}^{\text{sup}}=\sup_{u\in \mathcal{B}_1(\Omega)}F_{\Omega}(u).
$$
$J_{\Omega}^{\text{sup}}$ is defined in an analogous way, replacing $F$ by $J.$
If $x\in\Omegabar$ and the supremum is taken only over concentrating sequences, we write $F_{\Omega}^{\delta}(x),$ more precisely
$$
  F_{\Omega}^{\delta}(x)=\sup\left\{\limsup_{i\to\infty}F_{\Omega}(u_i)\,\Big|\quad \{u_i\}\subset \mathcal{B}_1(\Omega)\text{ concentrates at } x\right\}.
$$
We define in an analogous way $J_{\Omega}^{\delta}(x).$ If $\Omega=B_1$, then we define
$$
  F_{B_1,rad\searrow}^{\text{sup}}=\sup_{u\in W^{1,n}_{0,rad\searrow}(B_1)\cap \mathcal{B}_1(B_1)}F_{B_1}(u),
$$
$$
  F^{\delta}_{B_1,rad\searrow}(0)=\sup\left\{\limsup_{i\to\infty}F_{B_1}(u_i)\,\Big|\quad \{u_i\}\subset W_{0,rad\searrow}^{1,n}(B_1)\cap \mathcal{B}_1(B_1)\text{ concentrates at } 0\right\}.
$$
We define $J_{B_1,rad\searrow}^{\text{sup}}$ and $J^{\delta}_{B_1,rad\searrow}(0)$ in an analogous way.
\smallskip

$-$ If $\Omega\subset\re^n$ then $\Omega^{\ast}$ is its symmetric rearrangement, that is $\Omega^{\ast}=B_R(0),$ where $|\Omega|=R^n\omega_{n-1}/n .$ If $u\in W_0^{1,n}(\Omega),$ then $u^{\ast}\in W^{1,n}_{0,rad\searrow}(B_R(0))$ will denote the Schwarz symmetrization of $u.$ For basic propertis of the Schwarz symmetrization we refer to Kesavan \cite{Kesavan}, Chapters 1 and 2, which we will use throughout.
In particular we will use frequently and without further comment that if $u\in W_0^{1,2}(\Omega),$ then $u^{\ast}$ satisfies 
$$
  F_{B_R}(u)\leq F_{B_R}(u^{\ast})
  \quad\text{ if }\Omega=B_R
  \quad\text{ and }\quad \|\nabla u^{\ast}\|_{L^n(B_R)}\leq \|\nabla u\|_{L^n(\Omega)}.
$$
We will additionally  need, as in Flucher \cite{Flucher}, a slight modification of the Hardy-Littlewood, respectively P\'olya-Szeg\"o theorem, stated in the next proposition.  

\begin{proposition}
\label{prop:Hardy littlewoood modified}
(i) Lef $f\in L^p(\Omega)$ and $g\in L^q(\Omega),$ where $1/p+1/q=1.$ Then for any $a\in \re$
$$
  \int_{\{f\geq a\}}f\,g\leq \int_{\{f^{\ast}\geq a\}}f^{\ast}g^{\ast}.
$$
\smallskip

(ii) Let $u\in W^{1,n}_0(\Omega)$ such that $u\geq 0.$ Then for any $t\in (0,\infty)$
$$
  \int_{\{u^{\ast}\leq t\}}|\nabla u^{\ast}|^n\leq \int_{\{u\leq t\}}|\nabla u|^n
  \quad\text{ and }\quad \int_{\{u^{\ast}\geq t\}}|\nabla u^{\ast}|^n\leq \int_{\{u\geq t\}}|\nabla u|^n.
$$
\end{proposition}
\smallskip

\section{$n$-Green's function}
\label{section:n greens function}

If $x\in\Omega,$ then $G_{\Omega,x}$ will denote the $n$-Green's function of $\Omega$ with singularity at $x.$ It is the unique function defined on $\Omega\backslash\{x\}$ such that the principal value of the integral
\begin{equation}
 \label{eq:definition G Omega x}
  \intomega |\nabla G_{\Omega,x}(y)|^{n-2}\left\langle \nabla G_{\Omega,x}(y); \nabla\varphi(y)\right\rangle dy=\varphi(x)\quad\text{ for all }\varphi\in C_c^1(\Omega)
\end{equation}
and $G_{\Omega,x}=0$ on $\delomega.$ It can always be decomposed in the form
$$
  G_{\Omega,x}(y)=-\frac{1}{\cnn}\log(|x-y|)-H_{\Omega,x}(y),\qquad y\in\Omega\backslash\{x\},
$$
where $H_{\Omega,x}$ is a continuous  function on $\Omegabar$ and is $C^{1,\alpha}_{loc}$ in $\Omega \setminus \{x\}.$ This result is due to \cite{Kichenassamy Veron} (see therein Theorem stated in  (0.10) and (0.11) or Theorem 2.1 and in particular Remark 1.4). Another useful reference on the $n$-Green's function is \cite{Wang Wei}.

The conformal incenter $I_{\Omega}(x)$ of $\Omega$ at $x$ is defined by
$$
  I_{\Omega}(x)=e^{-\cnn H_{\Omega,x}(x)}.
$$
Before stating the next proposition we need the following definition. 
\smallskip

\begin{definition}
We say that a sequence of sets $\{A_i\}\subset \re^n$ are approximately small balls at $x\in\re^n$ (of radius $\tau_i$) as $i\to\infty$ if there exists  sequences $\tau_i,\sigma_i>0$ such that $\lim_{i\to\infty}\tau_i=0,$
$$
  \lim_{i\to\infty}\frac{\sigma_i}{\tau_i}=0
$$
and 
$$
  B_{\tau_i-\sigma_i}(x)\subset A_i\subset B_{\tau_i+\sigma_i}(x)\quad\text{ for all $i$ big enough.}
$$
\end{definition}

We will need the following properties of the $n$-Green's function. For what follows it is convenient to abbreviate, for $\beta\in [0,n)$
$$
  \alpha_{n,\beta}=(n-\beta)\omega_{n-1}^{1/(n-1)}.
$$

\begin{proposition}
\label{proposition:properties of Green's function} Let $x\in\Omega.$ Then
$G_{\Omega,x}$ and $I_{\Omega}(x)$ have the following properties:

(a) For every $t\in[0,\infty)$
$$
  \int_{\{G_{\Omega,x}<t\}}\left|\nabla G_{\Omega,x}(y)\right|^n dy=t.
$$

(b) For every $t\in [0,\infty)$
$$
  \int_{\{G_{\Omega,x}=t\}}\left|\nabla G_{\Omega,x}(y)\right|^{n-1} \dH(y)=1.
$$

(c)
$$
  \lim_{t\to\infty}\frac{n\,\left|\left\{G_{\Omega,x}>t\right\}\right|}{\omega_{n-1}e^{-n \cnn t}} = \left(I_{\Omega}(x)\right)^n.
$$

(d) If $B_R=\Omega^{\ast}$ is the symmetrized domain, then for any $x\in\Omega$ 
$$
  I_{\Omega}(x)\leq I_{B_R}(0)=R.
$$

(e) If $t_i\geq 0$ is a given sequence such that $t_i\to\infty$, then the sets $\{G_{\Omega,x}>t_i\}$ are approximately small balls at $x$ of radius $\tau_i$
$$
  \tau_i=I_{\Omega}(x)e^{-\cnn t_i}.
$$
In particular
$$
  \lim_{t\to\infty}\frac{n-\beta}{\omega_{n-1}e^{-\alpha_{n,\beta}t}}\int_{\{G_{\Omega,x}>t\}}|y-x|^{-\beta}=I_{\Omega}(x)^{n-\beta}.
$$
\end{proposition}

\begin{proof} In all statements we can assume without loss of generality that $x=0\in\Omega$ and abbreviate $G_{\Omega,x}=G,$ $H_{\Omega,x}=H.$
\smallskip

(a) By approximation \eqref{eq:definition G Omega x} holds also for any $\varphi\in W_0^{1,\infty}(\Omega).$ Thus (a) follows by taking $\varphi(y)=\min\{t,G(y)\}.$
\smallskip

(b) Observe that by (a) and the coarea formula
$$
  t=\int_{\{G<t\}}\left|\nabla G(y)\right|^{n-1} \left|\nabla G(y)\right| dy
  =
  \int_0^t\left(\int_{\{G=s\}}\left|\nabla G(y)\right|^{n-1}\dH(y)\right)ds.
$$
Hence, (b) follows from (a) by derivation.
\smallskip

(c) \textit{Step 1.} Write $G$ as
\begin{equation}
 \label{eq:G with H0 minus Hy}
  G(y)=-\frac{1}{\omega_{n-1}^{1/(n-1)}}\log(|y|)-H(0)+H(0)-H(y)
  =-\frac{1}{\omega_{n-1}^{1/(n-1)}}\log\left(\frac{|y|}{I_{\Omega}(0)}\right)+H(0)-H(y)
\end{equation}
Let $\|H\|_{\infty}=\sup \{|H(y)|:\, y\in\Omega\}$ and define
\begin{equation}
 \label{eq:def:S(t)}
  S(t)=\left\{y\in\Omega:\, |y|\leq I_{\Omega}(0) e^{-\omega_{n-1}^{1/(n-1)}(t-2\|H\|_{\infty})}\right\},
\end{equation}
and set
\begin{equation}
 \label{eq:def:m(t)}
  m(t)=\max_{y\in S(t)}|H(y)-H(0)|
\end{equation}
By the continuity of $H$ it holds that
\begin{equation}
 \label{eq:m(t) goes to 0}
  \lim_{t\to\infty}m(t)=0.
\end{equation}
We now define also the sets
\begin{align*}
 P(t)=& \left\{y\in\Omega:\, |y|<I_{\Omega}(0)e^{-\omega_{n-1}^{1/(n-1)}(t+m(t))}\right\}
 \smallskip
 \\
 Q(t)=& \left\{y\in\Omega:\, |y|<I_{\Omega}(0)e^{-\omega_{n-1}^{1/(n-1)}(t-m(t))}\right\},
\end{align*}
and claim that for all $t\geq 2\|H\|_{\infty}$
\begin{equation}
 \label{eq:P subset G geq t subset Q}
  P(t)\subset \{G>t\}\subset Q(t).
\end{equation}
Let $y\in P(t),$ then also $y\in S(t).$  So using \eqref{eq:G with H0 minus Hy}, \eqref{eq:def:m(t)} and finally the definition of $P(t)$ we get
\begin{align*}
  G(y)\geq & -\frac{1}{\omega_{n-1}^{1/(n-1)}}\log\left(\frac{|y|}{I_{\Omega}(0)}\right)-m(t)
  = -\frac{1}{\omega_{n-1}^{1/(n-1)}}\log\left(\frac{|y|}{I_{\Omega}(0)}\right)-\left(m(t)+t\right)+t > t.
\end{align*}
This shows $P(t)\subset\{G>t\}.$ If $y\in \{G>t\},$ then using again \eqref{eq:G with H0 minus Hy} we obtain
$$
  |y|<I_{\Omega}(0)e^{-\omega_{n-1}^{1/(n-1)}(t-H(0)+H(y))}.
$$
As $t-H(0)+H(y)\geq t-2\|H\|_{\infty}$ it holds that $y\in S(t)$ and hence $|H(0)-H(y)|\leq m(t).$ This implies that $y\in Q(t)$ and proves the claim \eqref{eq:P subset G geq t subset Q}. It now follows from \eqref{eq:P subset G geq t subset Q} that
\begin{align*}
  \frac{\omega_{n-1}}{n}I_{\Omega}(0)^ne^{-\omega_{n-1}^{1/(n-1)}(t+m(t))n}\leq |\{G>t\}|
  \leq \frac{\omega_{n-1}}{n}I_{\Omega}(0)^ne^{-\omega_{n-1}^{1/(n-1)}(t-m(t))n}
\end{align*}
Using \eqref{eq:m(t) goes to 0} proves (c).

\smallskip

(d) We refer to Flucher \cite{Flucher book} Lemma 8.2 page 64, or Lin \cite{Lin K C} Lemma 2.
\smallskip

(e) is deduced from \eqref{eq:P subset G geq t subset Q} in the following way: write 
$$
  I_{\Omega}(0)e^{-\omega_{n-1}^{1/(n-1)}(t+m(t))}=\tau(t)-\underline{\sigma}(t),\qquad  I_{\Omega}(0)e^{-\omega_{n-1}^{1/(n-1)}(t-m(t))}=\tau(t)+\overline{\sigma}(t),
$$
where
\begin{align*}
  &\tau(t)=I_{\Omega}(0)e^{-\omega_{n-1}^{1/(n-1)}t},
  \qquad 
  \underline{\sigma}(t)=I_{\Omega}(0)e^{-\omega_{n-1}^{1/(n-1)}t}
  \left(1-e^{-\omega_{n-1}^{1/(n-1)}m(t)}\right)
  \smallskip \\
  &\text{and }\quad 
  \overline{\sigma}(t)=I_{\Omega}(0)e^{-\omega_{n-1}^{1/(n-1)}t}
  \left(e^{\omega_{n-1}^{1/(n-1)}m(t)}-1\right)
\end{align*}
Using \eqref{eq:m(t) goes to 0}
$$
  \lim_{t\to\infty}\frac{\underline{\sigma}(t)}{\tau(t)}=\lim_{t\to\infty}\frac{\overline{\sigma}(t)}{\tau(t)}=0,
$$
and the first statement follows by setting $\tau_i=\tau(t_i),$ $\sigma_i=\max\{\underline{\sigma}(t_i),\overline{\sigma}(t_i)\}.$ The second statment follows from 
$$
  \int_{B_{\tau(t)-\sigma(t)}}|y|^{-\beta}dy\leq \int_{\{G>t\}}|y|^{-\beta}dy
  \leq 
  \int_{B_{\tau(t)+\sigma(t)}}|y|^{-\beta}dy,
$$
calculating explicitly the first and last integral, and using the first statement of (e).
\end{proof} 
\smallskip


We will need the following result.

\begin{lemma}\label{Volumeradii}
Let $\Omega$ be any smooth bounded domain of $\mathbb R^n$ and $x\in \Omega$. Suppose $\beta \in [0,n)$, then it holds
\begin{equation}\label{eq:Volumeradii}
\int_\Omega |y-x|^{-\beta} dy \geq \frac{\omega_{n-1}}{n-\beta} I_\Omega(x)^{n-\beta}.
\end{equation}
\end{lemma}

\begin{proof}
It is enough to prove \eqref{eq:Volumeradii} for $x =0$ and assume $0\in\Omega.$ We start the proof by recalling a sharp weighted isoperimetric inequality from \cite{Alvino} Theorem $1.1$ (ii).
\begin{equation}
 \label{eq:weightediso}
   \int_A |y|^{-\beta} dy \leq \frac1{\alpha_{n,\beta}} \left(\int_{\partial A} |y|^{-\frac{n-1}n\beta} d\mathcal{H}^{n-1}(y)\right)^{\frac{n}{n-1}}
   \quad\text{ for any smooth set $A\subset\re^n.$}
\end{equation}
Applying \eqref{eq:weightediso} to the set $\{G_{\Omega,0} >t\}$ and using H\"older inequality and Proposition \ref{proposition:properties of Green's function} (b), we have
\begin{align}
 \label{eq:VR1}
  \int_{\{G_{\Omega,0} >t\}} |y|^{-\beta} dy &\leq \frac1{\alpha_{n,\beta}} \left(\int_{\{G_{\Omega,0} =t\}} |\nabla G_{\Omega,0}(y)|^{\frac{n-1}n} \frac{|y|^{-\frac{n-1}n \beta}}{|\nabla G_{\Omega,0}(y)|^{\frac{n-1}n}} \dH(y)\right)^{\frac n{n-1}}
  \notag\\
  &\leq \frac1{\alpha_{n,\beta}} \left(\int_{\{G_{\Omega,0} =t\}} |\nabla G_{\Omega,0}(y)|^{n-1} \dH(y)\right)^{\frac1{n-1}} \left(\int_{\{G_{\Omega,0} =t\}} \frac{|y|^{-\beta}}{|\nabla G_{\Omega,0}(y)|} \dH(y)\right)
  \notag\\
  &= \frac1{\alpha_{n,\beta}} \int_{\{G_{\Omega,0} =t\}} \frac{|y|^{-\beta}}{|\nabla G_{\Omega,0}(y)|} \dH(y).
\end{align}
By co-area formula, we have
\[
\int_{\{G_{\Omega,0} >t\}} |y|^{-\beta} dy = \int_t^\infty \int_{\{G_{\Omega,0} =s\}} \frac{|y|^{-\beta}}{|\nabla G_{\Omega,0}(y)|} \dH(y)\, ds.
\]
Whence \eqref{eq:VR1} can be rewritten as
\[
\frac{d}{dt} \left(e^{\alpha_{n,\beta} t} \int_{\{G_{\Omega,0} >t\}} |y|^{-\beta} dy\right) \leq 0.
\]
In other words, $t \to e^{\alpha_{n,\beta} t} \int_{\{G_{\Omega,0} >t\}} |y|^{-\beta} dy$ is a non-increasing function. Using Proposition \ref{proposition:properties of Green's function} (e) we get
\[
\int_\Omega |y|^{-\beta} dy \geq \lim_{t\to \infty} e^{\alpha_{n,\beta} t} \int_{\{G_{\Omega,0} >t\}} |y|^{-\beta} dx = \frac{\omega_{n-1}}{n-\beta} I_\Omega(0)^{n-\beta},
\]
as wanted.
\end{proof}

The following proposition and its corollary are the main results of this section. We have included also the case $\beta=n,$ although this is not needed for the application to the singular Moser-Trudinger functional.

\begin{proposition}
\label{prop:Hypot H for r equal 1}
Let $\Omega\subset\re^n$ be any smooth bounded set and $x\in\Omega.$ Suppose $\beta\in [0,n],$ then it holds that
$$
  \omega_{n-1}^{n/(n-1)} I_{\Omega}(x)^{n-\beta}\leq \int_{\partial\Omega}\frac{|x-y|^{-\beta}}{\left|\nabla G_{\Omega,x}(y)\right|}\dH(y).
$$
\end{proposition}

\begin{proof} 
Let us first assume that $\beta\in [0,n).$ We can assume without loss of generality that $x=0\in\Omega$ and write again $G=G_{\Omega,x}.$ 
Applying \eqref{eq:VR1} to $t = 0$, we have
$$
  \int_{\Omega}|y|^{-\beta}dy\leq \frac{1}{\alpha_{n,\beta}}\intdelomega \frac{|y|^{-\beta}}{|\nabla G(y)|}\dH(y).
$$
It then follows from \eqref{eq:Volumeradii} that
$$
  \omega_{n-1}^{n/(n-1)}I_{\Omega}(0)^{n-\beta}=\alpha_{n,\beta}\frac{\omega_{n-1}}{n-\beta}I_{\Omega}(0)^{n-\beta}
  \leq
  \int_{\partial\Omega}\frac{|y|^{-\beta}}{\left|\nabla G_{\Omega,x}(y)\right|}\dH(y),
$$
which proves the proposition in the present case. The case $\beta=n$ is deduced by a continuity argument from the case $\beta<n.$
\end{proof}
\smallskip

For our application in Section \ref{section:ball to domain} the following extension to the level sets of $G_{\Omega,x}$ is crucial. This corollary generalizes Lemma 3 in Lin \cite{Lin K C} to the singular case $\beta\neq 0.$

\begin{corollary}
\label{corollary:Hypot H}
Let $n\geq 2,$ $\beta\in [0,n]$ and $\Omega$ be a bounded open  smooth subset of $\re^n$ with $x\in\Omega.$
Then all level sets $A_r$ of $G_{\Omega,x}$
$$
  A_r=\left\{y\in\Omegabar:\,G_{\Omega,x}(y)>-\frac{1}{\cnn}\log r\right\},\qquad r\in(0,1]
$$
satisfy the inequality
\begin{equation}
 \label{eq:hypothesis:level sets G for each r}
  \omega_{n-1}^{\frac{n}{n-1}}I_{\Omega}^{n-\beta}(x)
  \leq
  \frac{1}{r^{n-\beta}}
  \int_{\partial A_r}\frac{|x-y|^{-\beta}}{|\nabla G_{\Omega,x}(y)|}\dH(y).
\end{equation}
\end{corollary}

\begin{remark}
It can be shown that the inequality tends to an equality when $r\to 0,$ but this is not required for the proof of Theorem \ref{theorem:intro:Extremal for Singular Moser-Trudinger}. This follows from tha fact that $\lim_{y\to x}|y-x|\,|\nabla H_{\Omega,x}(y)|=0$ (see (1.2) in \cite{Kichenassamy Veron}) and therefore
$$
  |\nabla G_{\Omega,x}(y)|=\frac{1}{\omega_{n-1}^{1/(n-1)}|x-y|}(1+o(1))\quad\text{ as }y\to x.
$$
The proof is then similar to that of \cite{Lin K C} Lemma 1 (d) and Lemma 3, where it seems that it has been assumed that $\nabla H_{\Omega,x}$ is bounded near $x.$ The boundedness of $\nabla H_{\Omega,x}$ has been conjectured in \cite{Kichenassamy Veron} Remark 1.4, but we are not aware whether this has been proven.
\end{remark}

\begin{proof}
We can assume without loss of generality that $x=0\in\Omega.$ Apply Proposition \ref{prop:Hypot H for r equal 1} to the set $\Omega=A_r$. Note that $0\in A_r$ for all $r\in (0,1],$
$$
  G_{A_r,0}(y)=G_{\Omega,0}(y)+\frac{1}{\omega_{n-1}^{1/(n-1)}}\log r\quad\text{and}\quad 
  H_{A_r,0}(y)=H_{\Omega,0}(y)-\frac{1}{\omega_{n-1}^{1/(n-1)}}\log r.
$$
In particular this implies that $\nabla G_{A_r,0}=\nabla G_{\Omega,0}$ and $I_{A_r}(0)=r\, I_{\Omega}(0).$ This proves \eqref{eq:hypothesis:level sets G for each r}. 
\end{proof}

\section{Some Preliminary Results}\label{section:prelim. results}

We first note that it is sufficient to work with non-negative smooth maximizing sequences. More precisely we have the following lemma, which we will use in Section \ref{section:domain to ball} in a crucial way.

\begin{lemma}
\label{lemma new:sup over W12 same as over Cinfty}
Let $\{u_i\}\subset\mathcal{B}_1(\Omega)$ be a sequence such that the limit $\lim_{i\to\infty}F_{\Omega}(u_i)$ exists. Then there exists a sequence $\{w_i\}\subset \mathcal{B}_1(\Omega)\cap C_c^{\infty}(\Omega)$ such that 
$$
  \liminf_{i\to\infty}F_{\Omega}(w_i)\geq \lim_{i\to\infty}F_{\Omega}(u_i).
$$
Moreover, if $u_i$ concentrates at $x_0\in\Omegabar,$ then also $w_i$ concentrates at $x_0$. In particular maximizing sequences for $F_{\Omega}^{\sup}$ and $F_{\Omega}^{\delta}(x_0)$ can always be assumed to be smooth and non-negative.
\end{lemma}

\begin{proof}The proof is exactly the same as in the case $n=2,$ see Lemma 4 in \cite{Csato Roy Calc var Pde}.
\end{proof}
\smallskip

\begin{lemma}[compactness in interior]\label{lemma:compactness in the interior}
Let $0<\eta<1$ and suppose $\{u_i\}\subset W_0^{1,n}(\Omega)$ is such that
$$
  \limsup_{i\to\infty}\|\nabla u_i\|_{L^n}\leq \eta\quad\text{and}\quad 
  u_i\rightharpoonup u\text{ in }W^{1,n}(\Omega)
$$
for some $u\in W^{1,n}(\Omega).$ Then for some subsequence
$$
  \frac{e^{\alpha u_i^\frac{n}{n-1}}}{|x|^{\beta}}\to \frac{e^{\alpha u^\frac{n}{n-1}}}{|x|^{\beta}}\quad\text{ in }L^1(\Omega)
$$
and in particular $\lim_{i\to\infty}F_{\Omega}(u_i)=F_{\Omega}(u).$
\end{lemma}

\begin{proof} The idea of the proof is to apply Vitali convergence theorem.
We can assume that, up to a subsequence, that $u_i\to u$ almost everywhere in $\Omega$ and that
$$
  \|\nabla u_i\|_{L^n}\leq\theta=\frac{1+\eta}{2}<1\quad\forall\,i\in\mathbb{N}.
$$
We can therefore define $v_i=u_i/\theta\in \WBOmega,$
which satisfies $\|\nabla v_i\|_{L^n}\leq 1$ for all $i.$ Moreover let us define $\overline{\alpha}=\alpha\theta^{n/(n-1)}<\alpha,$ such that
$$
  \frac{\overline{\alpha}}{\alpha_n}+\frac{\beta}{n}<1.
$$
Let $E\subset\Omega$ be an arbitrary measurable set. We use H\"older inequality with exponents $r$ and $s,$ where
$$
  r=\frac{\alpha_n}{\overline{\alpha}}>1\quad\text{ and }\quad \frac{1}{s}=1-\frac{1}{r}> \frac{\beta}{n},
$$
to obtain that
$$
  \int_E\frac{e^{\alpha u_i^{\frac{n}{n-1}}}}{|x|^{\beta}}=
  \int_E\frac{e^{\overline{\alpha}v_i^{\frac{n}{n-1}}}}{|x|^{\beta}} \leq
  \left(\int_{E}e^{\alpha_n v_i^{\frac{n}{n-1}}}\right)^{\frac{1}{r}} \left(\int_E\frac{1}{|x|^{\beta s}}\right)^{\frac{1}{s}}.
$$
Let $\epsilon>0$ be given.
In view of the Moser-Trudinger inequality and using that $1/|x|^{\beta s}\in L^1(\Omega),$ we obtain that for any $\epsilon>0$ there exists a $\delta>0$ such that
$$
  \int_E\frac{e^{\alpha u_i^{\frac{n}{n-1}}}}{|x|^{\beta}}\leq \epsilon\quad\forall\,|E|\leq \delta\text{ and }i\in\mathbb{N}.
$$
This shows that the sequence $e^{\alpha u_i^{n/(n-1)}}/|x|^{\beta}$ is equi-integrable and the Vitali convergence theorem yields convergence in $L^1(\Omega).$
This proves the lemma.
\end{proof}
\smallskip

The proof of the next theorem can be found in Lions \cite{Lions}, Theorem I.6.

\begin{theorem}[Concentration-Compactness Alternative]\label{theorem:concentration alternative for singular moser trudinger}
Let $\{u_i\}\subset \mathcal{B}_1(\Omega).$ Then there is a subsequence and $u\in W_0^{1,n}(\Omega)$ with $u_i\rightharpoonup u$ in $W^{1,n}(\Omega),$  such that either

(a) $\{u_i\}$ concentrates at a point $x\in\Omegabar,$
\newline or

(b) the following convergence holds true
$$
  \lim_{i\to\infty}F_{\Omega}(u_i)=F_{\Omega}(u).
$$
\end{theorem}

The proof of the next two propositions is the same as in the $2$-dimensional case, see \cite{Csato Roy Calc var Pde}.

\begin{proposition}
\label{proposition:if u_i concentrates somewhere else than zero}
Let $\beta>0,$ $\{u_i\}\subset  \mathcal{B}_1(\Omega)$ and suppose that $u_i$ concentrates at $x_0\in\Omegabar,$ where $x_0\neq 0.$ Then 
one has that, for some subsequence, $u_i\rightharpoonup 0$ in $W^{1,n}(\Omega)$  and
$$
  \lim_{i\to\infty}F_{\Omega}(u_i)=F_{\Omega}(0)=0.
$$
In particular $F_{\Omega}^{\delta}(x_0)=0.$
\end{proposition}

\begin{remark}
\label{remark:F delta zero is not zero}
However, if $\alpha/\alpha_n+\beta/n=1$ then $F_{\Omega}^{\delta}(0)>0.$ To see this assume that $\epsilon>0$ is such that $B_{\epsilon}(0)\subset\Omega$ and define
$$
  u_i(x)=\left\{
  \begin{array}{rl}
  \displaystyle
   i&\text{ if }0\leq |x|\leq \epsilon e^{-\cnn i^{n/(n-1)}}
   \smallskip \\
   -\frac{1}{\cnn i^{1/(n-1)}}\log\left(\frac{|x|}{\epsilon}\right)
   &\text{ if } \epsilon e^{-\cnn i^{n/(n-1)}}\leq |x|\leq \epsilon
   \smallskip \\
   0 &\text{ if } |x|\geq \epsilon.
  \end{array}\right.
$$
One verifies by explicit calculation that $\{u_i\}\subset\mathcal{B}_1(\Omega)$ and it concentrates at $0.$ Let us show that $\liminf_{i\to\infty}F_{\Omega}(u_i)>0.$ First we make the estimate
$$
  F_{\Omega}(u_i)=\intomega\frac{e^{\alpha u_i^{\frac{n}{n-1}}}-1}{|x|^{\beta}}
  \geq 
  \int_{B_{a_i}(0)}\frac{e^{\alpha u_i^{\frac{n}{n-1}}}-1}{|x|^{\beta}},
  \quad\text{ where }\quad a_i=\epsilon e^{-\omega_{n-1}^{\frac{1}{n-1}}i^{\frac{n}{n-1}}}.
$$
To conclude it is sufficient to show that
$$
  \lim_{i\to\infty}\int_{B_{a_i}(0)}\frac{e^{\alpha u_i^{\frac{n}{n-1}}}-1}{|x|^{\beta}}=\omega_{n-1}\epsilon^{n-\beta}>0.
$$
We can use that  $|x|^{-\beta}$ is integrable and hence
$$
  \lim_{i\to\infty}\int_{B_{a_i}(0)}|x|^{-\beta}=0.
$$
One calculates that
$$
  \int_{B_{a_i}(0)}\frac{e^{\alpha u_i^{\frac{n}{n-1}}}}{|x|^{\beta}}
  =\omega_{n-1}\epsilon^{n-\beta}e^{i^{\frac{n}{n-1}}\left(\alpha-(n-\beta)\omega_{n-1}^{\frac{1}{n-1}} \right)}
  =\omega_{n-1}\epsilon^{n-\beta}\quad\text{ if }\frac{\alpha}{\alpha_n}+ \frac{\beta}{n}=1.
$$
This shows that $F_{\Omega}^{\delta}(0)>0.$
\end{remark}

We first prove Theorem \ref{theorem:intro:Extremal for Singular Moser-Trudinger} for some simple cases, which is the content of the next proposition. The proof is exactly the same as in the $2$-dimensional case, see \cite{Csato Roy Calc var Pde}.

\begin{proposition}
\label{proposition:sup attained if 0 notin Omegabar}
There exists $u\in \mathcal{B}_1(\Omega)$ such that $F_{\Omega}(u)=F_{\Omega}^{\sup}$ in the following cases:
$$
\text{(i)}\quad0\notin\Omegabar\quad\text{ or }\quad
 \text{(ii)}\quad\frac{\alpha}{\alpha_n}+\frac{\beta}{n}<1.
$$
\end{proposition}

\section{The Case $\Omega=B_1(0)$.}\label{section ball}

In this section we deal with the case where $\Omega$ is the unit ball. The following lemma is essentailly due to Adimurthi-Sandeep \cite{Adi-Sandeep}. There one can find a proof, which is similar to the $2$-dimensional case, see also \cite{Csato Roy Calc var Pde}.

\begin{lemma}
\label{lemma:properties of T_a}
Let $0<a<\infty,$ and $u$ be radial function on $B_1$. Define
$$
  T_a(u)(x)=a^{\frac{n-1}{n}}u\left(|x|^{\frac{1}{a}}\right).
$$
Then $T_a$ satisfies that 
$$
  T_a:W_{0,rad}^{1,n}(B_1)\to W_{0,rad}^{1,n}(B_1).
$$ 
$T_a$ is invertible with $(T_a)^{-1}=T_{1/a}$ and it satisfies
$$
  \|\nabla(T_a(u))\|_{L^n}=\|\nabla u\|_{L^n},\quad \forall\,u\in W_{0,rad}^{1,n}(B_1).
$$
Moreover if $a=1-\frac{\beta}{n},$ then
\begin{equation}\label{lemma:eq: transformation Adi-Sandeep}
  F_{B_1}(u)=\frac{1}{a}J_{B_1}(T_a(u)),\quad\forall\,u\in W_{0,rad}^{1,n}(B_1).
\end{equation}
\end{lemma}

The following corollary follows easily from Lemma \ref{lemma:properties of T_a}.

\begin{corollary}
\label{corollary:supremums for rad on D for F and G are same}
Let $a=1-\beta/n.$ Then the following identities hold true
$$
  \sup_{u\in \WBDiskrad}F_{B_1}(u)=\frac{1}{a}\sup_{u\in \WBDiskrad}J_{B_1}(u),
$$
and
$$
  F_{B_1}^{\text{sup}}=\frac{1}{a}J_{B_1}^{\text{sup}}.
$$
\end{corollary}

\begin{proof}
The first equality follows directly from Lemma \ref{lemma:properties of T_a}. By Schwarz symmetrization, the two equalities of the corollary are  equivalent.
\end{proof}
\smallskip

One of the crucial ingredients of the proof is the following result of Carleson and Chang \cite{Carleson-Chang}. Essential is the strict inequality in the following theorem. The second equality is an immediate consequence of the properties of Schwarz symmetrization.

\begin{theorem}[Carleson-Chang]
\label{theorem:Carleson and Chang strict inequality}
The following strict inequality holds true
$$
  J_{{B_1},rad\searrow}^{\delta}(0)<J_{B_1,rad\searrow}^{\sup}=J^{\sup}_{B_1}\,.
$$
\end{theorem}

\begin{remark}
The result in Carleson and Chang is acutally more precise, stating that
$$
  e^{1 + \frac{1}{2}+ \cdots + \frac{1}{n-1}}|B_1|=\sup_{x\in \overline{{B_1}}}J_{{B_1},rad\searrow}^{\delta}(x)<J_{B_1,rad\searrow}^{\sup}\,,
$$
but for our purpose we only need an estimate for the concentration level at $0.$ 
\end{remark}                    

From Lemma \ref{lemma:properties of T_a} and Theorem \ref{theorem:Carleson and Chang strict inequality} we easily deduce the following proposition.

\begin{lemma}
\label{lemma:strict ineq. between concentration level and supremum for sing. MT on disk}
Let $\{u_i\}\subset\mathcal{B}_1(B_1)$ be a sequence which concentrates at $0.$
If  $\{u_i^{\ast}\}$ also concentrates at $0,$ then the following strict inequality holds true
$$
  \limsup_{i\to\infty}F_{{B_1}}(u_i)<F_{B_1}^{\sup}\,.
$$
\end{lemma}

\begin{proof}The proof is identical to the $2$-dimensional case, see \cite{Csato Roy Calc var Pde}, with the only difference that one sets
 $a=1-\beta/n.$
\end{proof}
\smallskip

A consequence of Lemma \ref{lemma:strict ineq. between concentration level and supremum for sing. MT on disk} is the following theorem, stating that the supremum of $F_{B_1}$ is attained. The proof is here also identical to the $2$-dimensional case.

\begin{theorem}
\label{theorem:supremum of FOmega on Ball} The following strict inequality holds
$$
  F_{B_1}^{\delta}(0)<F_{B_1}^{\sup}.
$$
In particular there exists $u\in\mathcal{B}_1(B_1)$ such that $F_{B_1}^{\sup}=F_{B_1}(u).$
\end{theorem}

\section{Ball to Domain Construction}\label{section:ball to domain}

In view of Proposition \ref{proposition:sup attained if 0 notin Omegabar}, it remains to prove Theorem \ref{theorem:intro:Extremal for Singular Moser-Trudinger} for general domain with $0\in\Omega,$ when $\alpha/\alpha_n+\beta/n=1,$ and we can also take $\beta>0.$
Hence from now on we always assume that we are in this case. In addition, we assume in this section and Section \ref{section:domain to ball} that $0\in\Omega.$ The ball to domain construction is given by the following defnition: for $v\in W^{1,n}_{0,rad}(B_1)$  and $x\in\Omega,$ define $P_x(v)=u:\Omega\backslash\{x\}\to \re$ by
\begin{equation*}
  P_x(v)(y)=v\left(e^{-\omega_{n-1}^\frac{1}{n-1} G_{\Omega,x}(y)}\right)=v\left(\left(G_{B_1,0}\right)^{-1}(G_{\Omega,x}(y)\right),
\end{equation*}
where, by abuse of notation, we have identified $v$ and $G_{B_1,0}$ with the corresponding radial function, for instance:
$$
  G_{B_1,0}(z)=-\frac{1}{\cnn}\log z, \quad\text{ if }z\in (0,1].
$$

The main result of this section is the following theorem.

\begin{theorem}
\label{theorem:ball to general domain:sup inequality} Assume $\Omega\subset\re^n$ is a bounded open smooth set with $0\in\Omega.$
For any  $v\in W^{1,n}_{0,rad}(B_1)\cap \mathcal{B}_1(B_1)$ define $u=P_0(v).$ Then $u\in   \mathcal{B}_1(\Omega)$ and it satisfies
$$
  F_{\Omega}(u)\geq I_{\Omega}(0)^{n-\beta}F_{B_1}(v).
$$
In particular the following inequality holds true
$$
  F_{\Omega}^{\text{sup}}\geq I_{\Omega}(0)^{n-\beta}F^{\sup}_{B_1}.
$$
Moreover if $\{v_i\}\subset W_{0,rad}^{1,n}(B_1)$ concentrates at $0,$ then $u_i=P_0(v_i)$ concentrates at $0.$
\end{theorem}


The following lemma holds true for any domain, whether containing the origin or not. So we state this general version, although we will use it with $x=0.$

\begin{lemma}
\label{lemma:ball to domain via G:preserves dirichlet norm} Let $x\in\Omega$ and let $v\in W^{1,n}_{0,rad}(B_1).$
Then $P_x(v)\in W^{1,n}_0(\Omega)$ and in particular 
\begin{equation}\label{lemma:eq:nablu u is nabla Pu}
  \|\nabla (P_x(v))\|_{L^n(\Omega)}=\|\nabla v\|_{L^n(B_1)}.
\end{equation}
Moreover if $\{v_i\}\subset W_{0,rad}^{1,n}(B_1)$ concentrates at $0,$ then $P_x(v_i)$ concentrates at $x.$
\end{lemma}

\begin{proof} \textit{Step 1.} We write $G=G_{\Omega,x}.$
Let $h$ be defined by $h(y)=e^{-\cnn G(y)}$ and hence $u(y)=v(h(y)).$ In particular
$$
  \nabla u(y)=v'(h(y))\,\nabla h(y).
$$
Note that, since $G\geq 0$ in $\Omega$ we get that if $y\in h^{-1}({\{t\}})\cap \Omega,$ then $t\in [0,1].$ Thus the coarea formula gives that
\begin{align*}
 \intomega |\nabla u|^n=&\intomega |v'(h(y))|^n |\nabla h(y)|^{n-1}|\nabla h(y)| dy 
 \smallskip \\
 =&\int_0^1\left[\int_{h^{-1}({\{t\}})\cap \Omega}|v'(h(y))|^n|\nabla h(y)|^{n-1}\dH(y)\right]dt.
\end{align*}
Using that $|\nabla h(y)|=\cnn h(y)|\nabla G(y)|,$ gives
$$
  \intomega |\nabla u|^n=\int_0^1 \omega_{n-1} t^{n-1} |v'(t)|^n\left[\int_{h^{-1}({\{t\}})\cap\Omega}|\nabla G(y)|^{n-1}\dH(y)\right]dt.
$$
Note that $h^{-1}({\{t\}})\cap\Omega$ is a level set of $G.$
Thus we obtain from Proposition \ref{proposition:properties of Green's function} (b) that
$$
  \int_{h^{-1}({\{t\}})\cap \Omega}|\nabla G(y)|^{n-1} \dH(y)=1\quad\forall\,t\in(0,1),
$$
which implies that
$$
  \int_{\Omega}|\nabla u|^n=\int_0^1|v'(t)|^n \omega_{n-1} t^{n-1}\,dt=\int_{B_1}|\nabla v|^n.
$$
This proves \eqref{lemma:eq:nablu u is nabla Pu}.
\smallskip

\textit{Step 2.}
Let us now assume that $\{v_i\}$ concentrates at $0$ and let $\epsilon>0$ be given. We know from Proposition \ref{proposition:properties of Green's function} (e), that for some $M>0$ big enough $\{G>M\}\subset B_{\epsilon}(x).$ Thus we obtain exactly as in Step 1 that
$$
  \int_{\Omega\backslash B_{\epsilon}(x)}|\nabla u_i|^n\leq \int_{\{G\leq M\}}|\nabla u_i|^n=\int_{e^{-\cnn M}}^1|v_i'(t)|^n\, \omega_{n-1} t^{n-1}\,dt.
$$
The right hand side goes to $0,$ since $v_i$ concentrates. This proves that $u_i$ concentrates too.
\end{proof}

\smallskip

We are now able to prove the main theorem.

\begin{proof}[Proof (Theorem \ref{theorem:ball to general domain:sup inequality}).]
We abbreviate again $G=G_{\Omega,0}.$
From Lemma \ref{lemma:ball to domain via G:preserves dirichlet norm} we know that $u\in\mathcal{B}_1(\Omega).$ Using the coarea formula we get
\begin{align*}
 F_{\Omega}(u)=&\intomega\frac{\big(e^{\alpha u^{n/(n-1)}}-1\big)}{|y|^{\beta}}\frac{|\nabla G(y)|}{|\nabla G(y)|}dy
 =\int_0^{\infty}\left[\int_{G^{-1}(\{t\})\cap\Omega}\frac{(e^{\alpha u^{n/(n-1)}}-1)}{|y|^{\beta}|\nabla G(y)|}\dH(y)\right]dt
 \smallskip 
 \\
 =&\int_0^{\infty}\left(e^{\alpha v^{\frac{n}{n-1}}\left(e^{-\cnn t}\right)}-1\right)
 \left[\int_{G^{-1}(\{t\})\cap\Omega}\frac{1}{|y|^{\beta}|\nabla G(y)|}\dH(y)\right]dt.
\end{align*}
We now use Corollary \ref{corollary:Hypot H}, and set
$$
  r(t)=e^{-\cnn t},
$$
to obtain
\begin{align*}
  F_{\Omega}(u)\geq &I_{\Omega}(0)^{n-\beta}\int_0^{\infty}\frac{ e^{\alpha v^{\frac{n}{n-1}}(r(t))}-1}{r(t)^{\beta}}\, \omega_{n-1}\,\omega_{n-1}^{\frac{1}{n-1}} (r(t))^n dt 
  \smallskip \\
  =&-I_{\Omega}(0)^{n-\beta}\int_0^{\infty}\frac{ e^{\alpha v^{\frac{n}{n-1}}(r(t))}-1}{r(t)^{\beta}} \,\omega_{n-1} (r(t))^{n-1}\,r'(t) dt
  \smallskip 
  \\
  =&I_{\Omega}(0)^{n-\beta}\int_0^{1}\frac{ e^{\alpha v^{\frac{n}{n-1}}(r)}-1}{r^{\beta}}\, \omega_{n-1} r^{n-1}\,dr=I_{\Omega}(0)^{n-\beta}F_{B_1}(v).
\end{align*}
This proves the first claim of the theorem. 
The statement about the concentration follows directly from Lemma \ref{lemma:ball to domain via G:preserves dirichlet norm}.
\end{proof}

\section{Domain to Ball Construction}\label{section:domain to ball}

The aim of this section is to prove the inequality $F_{\Omega}^{\delta}(0)\leq I_{\Omega}(0)^{n-\beta}F_{B_1}^{\delta}(0)$. The main difficulty compared to the $2$-dimensional case is that we have to deal with the $n$-Laplace equation, which becomes a nonlinear partial differential equation with weaker regularity properties. We summarized the results on $n$-harmonic functions that we need in an Appendix. Recall that we assume $0\in\Omega.$ 

\begin{theorem}[Concentration Formula]
\label{theorem:concentration formula by domain to ball}
Suppose $\Omega$ contains the origin. Then the following formula holds
$$
  F_{\Omega}^{\delta}(0)= I_{\Omega}(0)^{n-\beta}F_{B_1}^{\delta}(0).
$$
\end{theorem}

The proof of this result will be a consequence of the following proposition, which allows to construct a concentrating sequence in the ball from a given concentrating sequence in $\Omega.$ 

\begin{proposition}\label{proposition:main domain to ball}
Let $\{u_i\}\subset  \mathcal{B}_1(\Omega)\cap C^{\infty}(\Omega)$ be a sequence which concentrates at $0$ and is a maximizing sequence for $F_{\Omega}^{\delta}(0).$
Then  there exists a sequence $\{v_i\}\subset W_{0,rad}^{1, n}(B_1)\cap \mathcal{B}_1(B_1)$ concentrating at $0,$ such that
\begin{equation*}
  F_{\Omega}^{\delta}(0)=\lim_{i\to\infty}F_{\Omega}(u_i)\leq I_{\Omega}^{n-\beta}(0)\liminf_{i\to\infty}F_{B_1}(v_i).
\end{equation*}
\end{proposition}

\begin{proof}[Proof (Theorem \ref{theorem:concentration formula by domain to ball}).]
From Lemma \ref{lemma new:sup over W12 same as over Cinfty} and Proposition \ref{proposition:main domain to ball} we immediately obtain that 
$$
  F_{\Omega}^{\delta}(0)\leq I_{\Omega}^{n-\beta}(0)F_{B_1}^{\delta}(0).
$$
The reverse inequality follows from Theorem \ref{theorem:ball to general domain:sup inequality}.
\end{proof}
\smallskip

The proof of Proposition \ref{proposition:main domain to ball} is long and technical. We split it into several intermediate steps. To make the presentation less cumbersome, we assume in what follows that $0\in\Omega.$ However, we actually need this, and the fact that concentration occurs at $0,$ only in Step 6 in the proof of Lemma \ref{main lemma si instead of 1. domain to ball}. The proof of the next two Lemmas is identical to the $2$-dimensional case, see \cite{Csato Roy Calc var Pde}.

\begin{lemma}
\label{lemma:limit over Bri same as Omega for subsequence}
Suppose $\{u_i\}\subset \mathcal{B}_1(\Omega)$ concentrates at $x_0\in\Omega$ and let $\{r_i\}\subset\re$ be  such that $r_i>0$ for all $i$ and $\lim_{i\to\infty}r_i=0.$ Then there exists a subsequence $u_{j_i}$ such that 
$$
  \lim_{i\to\infty}F_{\Omega}(u_i)=
  \lim_{i\to\infty}\int_{\Omega}\frac{e^{\alpha u_{i}^\frac{n}{n-1}}-1}{|x|^{\beta}}dx
  =\lim_{i\to\infty}\int_{B_{2r_i}(x_0)}\frac{e^{\alpha u_{j_i}^\frac{n}{n-1}}-1}{|x|^{\beta}}dx.
$$
Moreover any subsequence of $u_{j_i}$ will also satisfy the above equality.
\end{lemma}

\begin{lemma}
\label{lemma:a.e. convergence and bounded by 1}
Suppose $\{u_i\}$ is a sequence of measurable non-negative functions such that $u_i\to 0$ almost everywhere in $\Omega.$ Let $\{s_i\}\subset\re$ be a bounded sequence.
Then
$$
  \lim_{i\to\infty}\int_{\{u_i\leq s_i\}}\frac{e^{\alpha u_i^\frac{n}{n-1}}-1}{|x|^{\beta}}dx=0.
$$
\end{lemma}

\begin{lemma}
\label{lemma:the capacity argument}
Suppose $\{u_i\}\subset\mathcal{B}_1(\Omega)\cap C^{\infty}(\Omega)$ concentrates at $0\in\Omega$ and satisfies
\begin{equation}
 \label{eq:lemma:capacity argument maxim seq.}
  \lim_{i\to\infty}F_{\Omega}(u_i)=F_{\Omega}^{\delta}(0).
\end{equation}
Then for any $r>0$ there exists $j\in\mathbb{N}$ and $k_j\in [1,2]$ such that 
\begin{equation}
 \label{eq:lemma capacity arg. not empty Br}
  \{u_j\geq k_j\}\cap B_r(0)\neq \emptyset
\end{equation}
and
all connected components $A$ of $\{u_j\geq k_j\}$ will have the property:
\begin{equation}
 \label{eq:lemma:capacity arg. Br and B2r}
  \text{If }\;A\cap B_r(0)\neq\emptyset\quad\text{ then }\quad A\subset B_{2r}(0).
\end{equation}
Moreover $A$ has smooth boundary.
\end{lemma}

\begin{proof}  It is sufficient to prove that there exists $j\in \mathbb{N}$ such that  \eqref{eq:lemma capacity arg. not empty Br} is satisfied for $k_j=2$ and \eqref{eq:lemma:capacity arg. Br and B2r} holds for the connected components of  $\{u_j\geq 1\},$ i.e. $k_j=1.$  This implies that \eqref{eq:lemma capacity arg. not empty Br} and \eqref{eq:lemma:capacity arg. Br and B2r} also hold for any $k\in [0,1],$ and hence, using Sard's theorem, one can choose $k_j\in [1,2]$ appropriately such that $A$ has smooth boundary in addition.

First note that for all $m\in\mathbb{N}$ there exists a $j\geq m$ such that
\eqref{eq:lemma capacity arg. not empty Br} must hold. If this is not the case, then Lemma \ref{lemma:limit over Bri same as Omega for subsequence} and Lemma \ref{lemma:a.e. convergence and bounded by 1} imply that
$$
  \lim_{i\to\infty}F_{\Omega}(u_i)\leq \lim_{i\to \infty}\int_{B_r} \frac{e^{\alpha u_i^\frac{n}{n-1}}-1}\dxb 
  \leq\lim_{i\to\infty}\int_{\{u_i\leq 2\}}\frac{e^{\alpha u_i^\frac{n}{n-1}}-1}\dxb=0,
$$
which is a contradiction to \eqref{eq:lemma:capacity argument maxim seq.} (Recall that $F_{\Omega}^{\delta}(0)>0,$ see Remark \ref{remark:F delta zero is not zero}).

Suppose now that \eqref{eq:lemma:capacity arg. Br and B2r} does not hold. We show that this leads to a contradiction, using a capacity argument in dimension $n-1$ (following the idea in \cite{Flucher book} Equation (2.12) page 15). 
In that case there exists for all $j\in\mathbb{N}$ a connected component $D_j$ of $\{u_j\geq 1\}$ and $a,b\in\Omega$ such that
$$
  a\in D_j\cap B_r\quad\text{ and }\quad b\in D_j\cap \Omega\backslash B_{2r}.
$$
For what follows we fix $j$ and omit the explicit dependence on $j$ (Note that $a$ and $b$ depend on $j$). Without loss of generality we can assume, by rotating the domain, that  $b=(b_1,0)$ and $b_1\geq 2r.$
Therefore, since $D_j$ is connected, for all $x_1\in [r,2r]$ there exists a $X' = (x_2, \cdots, x_n) =X'(x_1)\in\re^{n-1}$ such that $x=(x_1, X'(x_1))\in D_j.$ In particular $u_j(x)\geq 1.$
Since $\Omega$ is bounded, there exists an $M>0,$ which is independent of the rotation of the domain (and hence of $j$), such that 
$\Omega\subset B_M(0)$. In particular this implies that
\begin{equation}\label{eq:X prime smaller M}
  \left|X'(x_1)\right|\leq M\quad\text{ for all }x_1\in [r,2r].
\end{equation}
Let us extend $u_j$ by zero in $\re^n\backslash\Omega.$  Denote by
$$
  B'_R(y')\text{ ball of radius $R$ in $\re^{n-1}$ centered at $y'\in \re^{n-1}$}\quad\text{ and }\quad \nabla'u=\left( \frac{\partial u}{\partial x_2},\ldots,\frac{\partial u}{\partial x_n}\right).
$$
With this notation, using \eqref{eq:X prime smaller M}, we have $u_j\left(x_1,y\right)=0$ for $y$ outside  of $B'_{2M}\left(X'(x_1)\right)$ for all  $x_1\in (r,2r).$ Moreover $u_j(x_1,X'(x_1))\geq 1.$ 
Using now the properties of $n$-capacity in $n-1$ dimension, see for instance \cite{Heinonen and others} Example 2.12 pages 35--36,
\begin{align*}
 \int_{B'_{2M}\left(X'(x_1)\right)}\left|\nabla'u_j(x_1,y')\right|^n dy'
 \geq &
 \operatorname{cap}_n\left(\left\{X'(x_1)\right\},B_{2M}'\left(X'(x_1)\right)\right)
 \smallskip \\
 =& \operatorname{cap}_n\left(\left\{0\right\},B_{2M}'\left(0\right)\right)=c(n,M)>0,\quad\text{ for all }x_1\in (r,2r),
\end{align*}
for some positive constant $c(n,M)$ depending only on $n$ and $M.$ Hence, using also that $\Omega$ intersected with any plane where first coordinate equals $x_1$ is containded in $\{x_1\}\times B'_{M}(0),$ we get
\begin{align*}
 \int_{\Omega\backslash B_r(0)}|\nabla u_j|^n
 \geq &
 \int_{\Omega\backslash B_r(0)}|\nabla'u_j|^n
 \geq \int_r^{2r}\int_{\Omega\cap \{y_1=x_1\}}|\nabla'u_j(x_1,y')|^n dy'\, dx_1
 \smallskip \\
 =&
 \int_r^{2r}\int_{B'_{M}(0)}|\nabla'u_j(x_1,y')|^n dy'\, dx_1
 =  \int_r^{2r}\int_{B'_{2M}(X'(x_1))}|\nabla'u_j(x_1,y')|^n dy'\, dx_1
 \smallskip \\
 \geq & c(n,M)r.
\end{align*}
This implies that 
$$ 
  r\,c(n,M) \leq \int_{\Omega\backslash B_r}|\nabla u_j|^n.
$$
But this cannot hold true for all $j,$ since $u_j$ concentrates at $0$.
\end{proof}
\smallskip

The next lemma is about the first modification of the sequence $\{u_i\}$ given in Proposition \ref{proposition:main domain to ball}.
\begin{lemma}
\label{lemma:main properties step 1-3 in Flucher}
Let $\{u_i\}\subset \mathcal{B}_1(\Omega)\cap C^{\infty}(\Omega) $ be a sequence which concentrates at $0\in\Omega$ and satisfies
$$
  \lim_{i\to\infty}F_{\Omega}(u_i)=F_{\Omega}^{\delta}(0).
$$
Then there exists  a sequence $\{v_i\}\subset \mathcal{B}_1(\Omega)$ and sequences $r_i>0,$ with $r_i\to 0$  and  $\{k_i\}\in [1,2]$
 such that 
\begin{align*}
 \{v_i\geq k_i\}\subset B_{2r_i},\quad \Delta_n v_i=0\quad\text{in }\{v_i<k_i\}.
\end{align*}
Moreover $v_i$ has the properties: there exist a sequence $\{\lambda_i\}\subset\re,$ $\lambda_i>0$ such that  
\begin{align*}
 &\text{(i)}\quad \lim_{i\to\infty}\lambda_i=\infty 
 \smallskip \\
 &\text{(ii)}\quad \lim_{i\to\infty}v_i(y)=0\quad\text{ for all $y$ in }\Omega\backslash\{0\} 
 \smallskip \\
 &\text{(iii)}\quad \lambda_i v_i\to G_{\Omega,0}\quad\text{in }C^1_{\text{loc}}(\Omega\backslash \{0\})
 \smallskip \\
 &\text{(iv)} \quad\lim_{i\to\infty}F_{\Omega}(v_i)=F_{\Omega}^{\delta}(0).
\end{align*}
\end{lemma}

\begin{proof}
\textit{Step 1.} Take a sequence of positive real numbers $r_i$ such that $\lim_{i\to\infty}r_i=0$ and choose a subsequence of $u_i$, using Lemma \ref{lemma:limit over Bri same as Omega for subsequence}, such that
\begin{equation}
 \label{eq:proof:lemma st13 Bri}
  F_{\Omega}^{\delta}(0)=\lim_{i\to\infty}F_{\Omega}(u_i)=\lim_{i\to\infty} \int_{B_{r_i}}\frac{e^{\alpha u_i^\frac{n}{n-1}}-1}\dxb.
\end{equation}
Choosing again a subsequence we can assume by Lemma \ref{lemma:the capacity argument} that there exist $k_i\in [1,2]$ such that all connected components $A$ of $\{u_i\geq k_i\}$ which intersect $B_{r_i}$ are contained in $B_{2r_i}.$ We define $A_i$ as the union of all such $A.$ We also know from Lemma \ref{lemma:the capacity argument} that $A_i$ is not empty. Let $w_i\in W^{1,n}(\Omega\backslash\overline{A}_i)$ be the solution of, see Theorem \ref{P2, Existence and regularity},
$$
  \left\{\begin{array}{c}
          \Delta_n w_i=0\quad\text{in }\Omega\backslash \overline{A}_i 
          \smallskip \\
          w_i=0\quad\text{on }\delomega,\quad w_i=k_i\quad\text{on }\partial A_i.
         \end{array}\right.
$$
We now define $v_i\in W^{1,n}_0(\Omega)$ as
$$
  v_i=\left\{\begin{array}{rl}
               u_i&\text{ in }A_i \smallskip \\
               w_i&\text{ in }\Omega\backslash A_i\,.
              \end{array}\right.
$$
Since $n$-harmonic functions minimize the $n$-Dirichlet integral we have $\|\nabla v_i\|_{L^n}\leq \|\nabla u_i\|_{L^n}.$ 
Thus we have constructed a sequence which has the properties: (we have used Theorem \ref{P1, Strong maximum principle} in the second property)
\begin{align*}
 \{v_i\geq k_i\}\subset B_{2r_i},\quad \Delta_n v_i=0\quad\text{in }\{v_i<k_i\}\quad\text{and}\quad
 \|\nabla v_i\|_{L^n}\leq 1.
\end{align*}

\textit{Step 2.}  We will show in this Step that for all $y\in\Omega\backslash\{0\}$ we have $v_i(y)>0$ for all $i$ large enough and $\lim_{i\to\infty}v_i(y)=0.$
The fact that $v_i(y)>0$ follows from the maximum principle Theorem \ref{P1, Strong maximum principle}.
Since $\Omega$ is bounded there exists $M>0$ such that $\Omegabar\subset B_M.$ Define $W_i=B_M\backslash\overline{B_{2r_i}}$ and let $\psi_i$ be the solution of 
$$
  \left\{\begin{array}{c}
          \Delta_n \psi_i=0\quad\text{ in }W_i 
          \smallskip \\
          \psi_i=2\;\text{ on }\partial B_{2r_i}\quad\text{ and }
          \quad\psi_i=0\;\text{ on }\partial B_M.
         \end{array}\right.
$$
The function $\psi_i$ can be given explicitly:
$$
  \psi_i=\frac{2}{\log\big(\frac{2r_i}{M}\big)}\log\left(\frac{|x|}{M}\right).
$$
Recall that $k_i\in[1,2]$ and note that
\begin{align*}
 &\psi_i>0\,\text{ and }\,v_i=0\quad\text{ on }\delomega,
 \smallskip \\
 &\psi_i=2\,\text{ and }\,v_i<k_i\leq 2\quad\text{ on }\partial B_{2r_i},
\end{align*}
and thus $\psi_i-v_i>0$ on $\partial W_i$. Since $v_i$ is also harmonic in $W_i$ the comparison principle (Theorem \ref{P3, Comparison Principle}) implies that $v_i\leq \psi_i$ in $W_i$.
For $i$ big enough $y\in W_i$ and the claim of Step 2 follows from the fact that $\lim_{i\to\infty}\psi_i(y)=0.$ 
\smallskip 

\textit{Step 3.} Choose $y\in \Omega\backslash\{0\}$ and define $\lambda_i$ by
\begin{equation}\label{eq:lemma proof:def of lambdai by uiy}
  \lambda_i=\frac{G_{\Omega,0}(y)}{v_i(y)}\quad\Leftrightarrow\quad \lambda_iv_i(y)=G_{\Omega,0}(y)
\end{equation}
In view of Step 2 this is well defined, $\lambda_i>0$ and
$$
  \lim_{i\to\infty}\lambda_i=\infty.
$$
Let $y\in K_1\subset \Omega\backslash\{0\}$ be a compact set. Choose another compact set $K_2,$ such that $K_1\subset\subset K_2\subset \Omega\backslash\{0\}.$
Applying Harnack inequality (Theorem \ref{P4, Harnack inequality}) on $K_2$ we get that there exist $c_1,c_2>0,$ such that 
\begin{equation}
 \label{vi uniformly bounded}
  c_1|G_{\Omega,0}(y)|\leq |\lambda_i v_i(x)|\leq c_2|G_{\Omega,0}(y)|,\quad\forall\,x\in K_2\,\text{ and }\,\forall\,i\text{ large enough}.
\end{equation}
Hence the sequence $\lambda_iv_i$ is uniformly bounded in the $C^0(K_2)$ norm. It follows from Theorem \ref{P5, Uniform Regularity} that 
\begin{equation}
 \label{vi uniformly bounded in C1 alpha norm}
  \lambda_i v_i\quad\text{ is uniformly bounded in the $C^{1,\alpha}(K_1)$ norm}
\end{equation}
for some $0<\alpha.$ Using the compact embedding $C^{1,\alpha}(K_1)\hookrightarrow C^1(K_1)$ we obtain that there exists $g\in C^1(K_1)$ and a subsequence $v_i$ with
$$
  g_i:=\lambda_i v_i\to g\quad\text{ in }C^1(K_1).
$$ 
It follows from \eqref{eq:lemma proof:def of lambdai by uiy} and Corollary \ref{P6, Bocher Theorem} that $g=G_{\Omega,0},$ once we have shown that $g=0$ on $\delomega.$ We prove this claim in the next step.
\smallskip

\textit{Step 4.} We show now that $g=0$ on $\delomega.$ Define $\Omega_{\epsilon}$  as 
$$
  \Omega_{\epsilon}=\left\{x\in\Omega:\, 0<\operatorname{dist}(x,\delomega)<\epsilon\right\},
$$
where $\epsilon$ will be chosen later small enough, and
$$
  \partial\Omega_{\epsilon}=\delomega\cup \Gamma_{\epsilon}\quad\text{ where }\quad \Gamma_{\epsilon}=\left\{ x\in \Omega:\, \operatorname{dist}(x,\delomega)=\epsilon\right\}.
$$

\textit{Step 4.1.} We claim that $g_i$ are uniformly bounded on $\Omega_{\epsilon}.$
Note that $g_i=0$ on $\delomega.$ So for small enough $\epsilon$ it follows from Lemma \ref{lemma1:Moser iteration} and Remark \ref{remark:locality Lieberman hypothesis} (ii) (as in the proof of Proposition \ref{proposition:hypothesis Liebermann satisfied}) that
\begin{equation}
 \label{gi Linf norm bound by Ln}
  \|g_i\|_{L^{\infty}(\Omega_{\epsilon})}\leq C(\Omega_{\epsilon})\|g_i\|_{L^n(\Omega_{\epsilon})}.
\end{equation}
We now fix such an $\epsilon>0$ for which \eqref{gi Linf norm bound by Ln} holds and we can also assume that $\Omega_{\epsilon}$ is a smooth set. It follows from \eqref{vi uniformly bounded in C1 alpha norm} that there exists $\Lambda_1=\Lambda_1(K_1,K_2)$ (chosing $K_1$ such that $\Gamma_{\epsilon}\subset K_1$), such that
$$
  \|g_i\|_{C^{1,\alpha}(\Gamma_{\epsilon})}\leq \Lambda_1\quad\text{ for all $i$ big enough},
$$
and hence also for some $\Lambda_2>0$
$$
  \|g_i\|_{W^{1-\frac{1}{n},n}(\Gamma_{\epsilon})}\leq \Lambda_2\quad\text{ for all $i$ big enough}.
$$
Chose now a bounded right inverse $T$ of the trace operator on $\delomega_{\epsilon}$ as
$$
  T:W^{1-\frac{1}{n}}(\partial\Omega_{\epsilon})\to W^{1,n}(\Omega_{\epsilon})
$$
and apply it to $g_i$ restricted to $\delomega_{\epsilon}.$ Hence there exists $h_i\in W^{1,n}(\Omega_{\epsilon}),$ $h_i=T(g_i)$ such that 
$$
  h_i=g_i\quad\text{ on }\delomega_{\epsilon}
$$
and 
\begin{equation}
 \label{hi uniformly bounded}
  \|h_i\|_{W^{1,n}(\Omega_{\epsilon})}\leq C_1(\Omega_{\epsilon})\|g_i\|_{W^{1-\frac{1}{n},n}(\delomega_{\epsilon})}\leq C_1(\Omega_{\epsilon})\Lambda_2\,.
\end{equation}
As in the proof of Proposition \ref{proposition:hypothesis Liebermann satisfied}, since $\Delta_n g_i=0$ in $\Omega_{\epsilon}$ for $i$ big enough,
\begin{align*}
 \|g_i-h_i\|_{L^n(\Omega_{\epsilon})}\leq C_2(\Omega_{\epsilon})\|\nabla g_i-\nabla h_i\|_{L^n(\Omega_{\epsilon})}
 \leq
 C_2\left(\|\nabla g_i\|_{L^n(\Omega_{\epsilon})}+\|\nabla h_i\|_{L^n(\Omega_{\epsilon})}\right) 
 \smallskip \\
 \leq 2 C_2 \|\nabla h_i\|_{L^n(\Omega_{\epsilon})}\leq 2 C_1 C_2\Lambda_2\,.
\end{align*}
Since the $h_i$ are also uniformly bounded by \eqref{hi uniformly bounded} it follows the $g_i$ are uniformly bounded in the $L^{\infty}(\Omega_{\epsilon})$ norm. This shows Step 4.1.
\smallskip

\textit{Step 4.2.} We now conclude that $g=0$ on $\delomega.$
Fix some $a>0$ so that $B_a(0)\subset\subset\Omega$ and define $\Omega_a=\Omega\backslash \overline{B_a}(0).$ 
Note that $g_i$ is uniformly bounded on $\partial B_a(0)$ in the $C^{1,\alpha}$ norm, using again \eqref{vi uniformly bounded in C1 alpha norm}, i.e. for some $\Lambda_3>0$ we have
$$
  \|g_i\|_{C^{1,\alpha}(\partial B_a(0))}\leq \Lambda_3\,.
$$
On the compact set $\Omega\backslash (B_a(0)\cup\Omega_{\epsilon})$ $g_i$ is uniformly bounded by \eqref{vi uniformly bounded}. Thus, together with Step 4.1 this shows that there exists a constant $M_0$ independent of $i$ such that
$$
  \|g_i\|_{L^{\infty}(\Omega_a)}\leq M_0\,.
$$
Note that
$$
  \Delta_n g_i=0\quad\text{ in }\Omega_a\,.
$$
So it follows from Theorem \ref{P2, Existence and regularity} that
$$
  \|g_i\|_{C^{1,\alpha}(\overline{\Omega_a})}\leq C \left(M_0,\Lambda_3\right).
$$
It follows that for some subsequence $g_i\to g$ in $C^1(\overline{\Omega_a})$ from which it follows that $g=0$ on $\delomega.$
\smallskip

\textit{Step 5.} It remains to prove (iv). Recall that $v_i\leq k_i$ in $\Omega\backslash A_i$. We therefore obtain, using Lemma \ref{lemma:a.e. convergence and bounded by 1} twice and the definition of $A_i$ that
\begin{align*}
 &\lim_{i\to\infty}\intomega\frac{e^{\alpha v_i^\frac{n}{n-1}}-1}\dxb
 =\lim_{i\to\infty}\int_{A_i}\frac{e^{\alpha u_i^\frac{n}{n-1}}-1}\dxb
 \geq \lim_{i\to\infty}\int_{A_i\cap B_{r_i}}\frac{e^{\alpha u_i^\frac{n}{n-1}}-1}\dxb
  \smallskip \\
 =&\lim_{i\to\infty}\int_{ \{u_i\geq k_i\}\cap B_{r_i}}\frac{e^{\alpha u_i^\frac{n}{n-1}}-1}\dxb=\lim_{i\to\infty}\int_{B_{r_i}}\frac{e^{\alpha u_i^\frac{n}{n-1}}-1}\dxb=F_{\Omega}^{\delta}(0),
\end{align*}
where we have used \eqref{eq:proof:lemma st13 Bri} in the last equality.
\end{proof}
\smallskip

The next lemma is about the second modification of the sequence $\{u_i\}$ given in Proposition \ref{proposition:main domain to ball}, following the first modification given by Lemma \ref{lemma:main properties step 1-3 in Flucher}.

\begin{lemma}
\label{lemma:the one we sent to Struwe}
Let $\{u_i\}\subset W_0^{1,n}(\Omega)$ be a sequence and $\lambda_i$ a sequence in $\re$ such that $\lambda_i\to\infty,$  
$$
  \lambda_i u_i\to G_{\Omega,0}\quad\text{ in }C^0_{loc}({\Omega}\backslash\{0\})\quad\text{ and }\quad \Delta_n u_i=0\text{ in }\{u_i<1\}.
$$
Then there exists a subsequence $\lambda_{i_l}$ and a sequence $\{v_l\}\subset W^{1,n}_0(\Omega)$ such that the following properties hold true:

(a) $\lambda_{i_l}\geq l$ 

(b) The sets $\left\{v_l\geq  l/\lambda_{i_l}\right\}$
are approximately small balls at $0$ as $l\to\infty.$

(c) $v_l(x)\to 0$ as $l\to\infty$ for every $x$ in $\Omega\backslash\{0\}.$

(d) For every $l$
$$
  \intomega|\nabla v_l|^n\leq\intomega|\nabla u_{i_l}|^n.
$$

(e) The inequality $v_l\geq u_{i_l}$ holds in $\Omega.$ In particular  
$F_{\Omega}( v_l)\geq F_{\Omega} ( u_{i_l}).$
\end{lemma}

\begin{proof}
The proof is exactly the same as in the $2$-dimensional case, since one can use the strong maximum principle, Theorem \ref{P1, Strong maximum principle}, for $n$-harmonic functions.
\end{proof}
\smallskip

After having modified the sequence $\{u_i\}$ given in Proposition \ref{proposition:main domain to ball} in the two previous lemmas, we finally construct the appropriate corresponding sequence $\{v_i\}\subset W_{0,rad}^{1,n}(B_1).$  This is contained in the following lemma.

\begin{lemma}\label{main lemma si instead of 1. domain to ball}
Let $\{u_i\}\subset W_0^{1,n}(\Omega)$ and $\{s_i\}\subset\re$ be sequences with the following properties:
$$
  s_i\leq 1\quad\forall\,i\in\mathbb{N},
$$
the sets $\{u_i\geq s_i\}$ are approximately small balls at $0$ as $i\to\infty$ and
moreover suppose that pointwise $u_i(x)\to 0$ for all $x\in\Omega\backslash\{0\}.$
Then there exists a sequence $\{v_i\}\subset W^{1,n}_{0,rad}(B_1)$ such that for all $i$
$$
  \|\nabla v_i\|_{L^n(B_1)}\leq \|\nabla u_i\|_{L^n(\Omega)}
$$
and, assuming that the left hand side limit exists, 
$$
  \lim_{i\to\infty}F_{\Omega}(u_i)\leq 
  I_{\Omega}(0)^{n-\beta}\lim\inf_{i\to\infty}F_{B_1}(v_i).
$$
Moreover $v_i(x)\to 0$ for all $x\in B_1\backslash \{0\}$ and if $v_i$ concentrates at some $x_0\in B_1,$ then $x_0=0.$
\end{lemma}


\begin{proof}  Throughout this proof $G=G_{\Omega,0}$ shall denote the $n$-Green's function of $\Omega$ with singularity at $0.$
Recall that by assumption there exists real positive numbers $\rho_i$ and $\epsilon_i$ such that for $i\to\infty$
 \begin{equation}
  \label{eq:lemma si instead of 1:eps over rho to zero}
   \rho_i\to 0\quad\text{ and }\quad \frac{\epsilon_i}{\rho_i}\to 0,
 \end{equation}
satisfying for all $i$ the following inclusion
\begin{equation}
 \label{eq:lemma si instead of 1:ui bigger si are asymptotic balls}
  B_{\rho_i-\epsilon_i}\subset\{u_i\geq s_i\}\subset B_{\rho_i+\epsilon_i}.
\end{equation}
\smallskip

\textit{Step 1.} 
Let us define $\lambda_i$, implicitly,  by the following equation: 
\begin{equation}
 \label{eq:lemma:definition of rhoi via conf. incenter. dim n}
  \rho_i=I_{\Omega}(0)e^{-\cnn \lambda_i},
\end{equation}
that is
$$
  \lambda_i=-\frac{1}{\cnn}\log\left(\frac{\rho_i}{I_{\Omega}(0)}\right).
$$
Note that $\lambda_i\to\infty$ as $i\to\infty.$ We claim that there exists $t_i\geq \lambda_i$ such that 
\begin{equation}
 \label{eq:proof:ti-lambdai goes to zero. dim n}
 \lim_{i\to\infty}(t_i-\lambda_i)=0
\end{equation}
and
\begin{equation}
 \label{eq:proof:G bigger ti in ui bigger 1. dim n}
  \{G\geq t_i\}\subset\{u_i\geq s_i\}.
\end{equation}
To show this we use Proposition \ref{proposition:properties of Green's function} (e), which states that if $t_i\geq 0$ is given such that $t_i\to\infty$, then there exists $\sigma_i\geq 0$ such that
$$
  \lim_{i\to\infty}\frac{\sigma_i}{\tau_i}=0
$$
and 
$$
  B_{\tau_i-\sigma_i}\subset\{G\geq t_i\}\subset B_{\tau_i+\sigma_i},
$$
where $\tau_i=I_{\Omega}(0)e^{-\cnn t_i}.$ In view of \eqref{eq:lemma si instead of 1:ui bigger si are asymptotic balls} it is therefore sufficient to choose $t_i$ such that
\begin{equation}
 \label{eq:proof lemma:implicit def. of ti. dim n}
  \tau_i+\sigma_i=\rho_i-\epsilon_i\,.
\end{equation}
It remains to show that with this choice \eqref{eq:proof:ti-lambdai goes to zero. dim n} is also satisfied. Using \eqref{eq:lemma:definition of rhoi via conf. incenter. dim n} and solving the previous equation for $t_i$ explicitly gives that
$$
  t_i=\lambda_i-\frac{1}{\cnn}\log\left(1-\frac{\epsilon_i+\sigma_i}{\rho_i}\right).
$$
Since we know from \eqref{eq:lemma si instead of 1:eps over rho to zero} that 
$\epsilon_i/\rho_i\to0,$ it is sufficient to show that $\sigma_i/\rho_i\to 0.$ We obtain from \eqref{eq:proof lemma:implicit def. of ti. dim n} that
$$
  \frac{\sigma_i}{\tau_i}=\frac{\sigma_i}{\rho_i-\epsilon_i-\sigma_i}
  =\frac{\sigma_i}{\rho_i\left(1-\frac{\epsilon_i}{\rho_i}-\frac{\sigma_i}{\rho_i} \right)}.
$$
Solving this equation for  $(\sigma_i/\rho_i)$ and using that $\epsilon_i/\rho_i\to 0$ and $\sigma_i/\tau_i\to 0$ shows that also $(\sigma_i/\rho_i)\to 0.$ This proves \eqref{eq:proof:ti-lambdai goes to zero. dim n}.

\smallskip

\textit{Step 2.} In this step we will show that
\begin{equation}
 \label{eq:proof lemma: step 2 grad norm of ui and ti. dim n}
  \int_{\{u_i<1\}}|\nabla u_i|^{n}\geq \frac{s_i^n}{t_i^{n-1}}.
\end{equation}
Let us denote
$$
  U=\{u_i\geq s_i\}\quad\text{ and }\quad V=\{G\geq t_i\}.
$$
From Step 1 we know that $V\subset U$ and since $u_i=0$ on $\delomega$ we also have that $\overline{U}\subset \Omega.$ Let $h_i\in W^{1,n}(\Omega\backslash V)$ be the unique solution of the problem
\begin{align*}
  \Delta_n h_i=&0\quad\text{in }\Omega\backslash V 
  \smallskip \\
  h_i=0\quad\text{on }\delomega\quad\text{and}&\quad h_i=1\quad\text{on }  \partial V.
\end{align*}
We see that this is satisfied precisely by $h_i=G/t_i$. Let us define $w_i\in W^{1,n}(\Omega\backslash V)$ by
$$
  w_i=\left\{\begin{array}{lr}
              \frac{u_i}{s_i} &\text{ in }\Omega\backslash U 
              \smallskip \\
              1 &\text{ in } \overline{U}\backslash V.
             \end{array}\right.
$$
Note that $w_i$ has the same boundary values as $h_i$ on the boundary of $\Omega\backslash V.$ Since $h_i$ is the unique minimizer of the functional 
$$
  J(h)=\int_{\Omega\backslash V}|\nabla h|^n
$$
among all functions with these fixed boundary values, we get that
$$
  \int_{\Omega\backslash V}|\nabla h_i|^n\leq \int_{\Omega\backslash V}|\nabla w_i|^n=\int_{\Omega\backslash U}|\nabla w_i|^n=
  \int_{\{u_i<s_i\}}|\nabla u_i|^n.
$$
From Proposition \ref{proposition:properties of Green's function} (a) we know that
$$
  \int_{\Omega\backslash V}|\nabla h_i|^n =\int_{\{G<t_i\}}\left|\nabla\left(\frac{G}{t_i}\right)\right|^n=\frac{1}{t_i^{n-1}}.
$$
Setting this into the previous inequality proves \eqref{eq:proof lemma: step 2 grad norm of ui and ti. dim n}.

\smallskip

\textit{Step 3.} In this step we will define $v_i\in W_{0,rad}^{1,n}(B_1).$ 
Let $\Omega^{\ast}=B_R$ be 
the symmetrized domain and $u_i^{\ast}\in W_{0,rad}^{1,n}(B_R)$ be the radially decreasing symmetric rearrangement of $u_i$. Then there exists $0<a_i<R$ such that
\begin{equation}
 \label{eq:proof lemma step 3:definition of ai. dim n}
 \{u_i^{\ast}\geq s_i\}=B_{a_i}.
\end{equation}
Moreover define $0<\delta_i<1$ by 
$$
  \delta_i=e^{-\cnn t_i}.
$$
At last we can define $v_i$ as
$$
  v_i(x)=\left\{\begin{array}{lr}
                 -\frac{s_i}{\cnn t_i}\log(|x|) &\text{ if }x\geq \delta_i
                 \smallskip \\
                 u_i^{\ast}\left(\frac{a_i}{\delta_i}x\right) & \text{ if }x \leq \delta_i\,.
                \end{array}\right.
$$
Note that $v_i$ belongs indeed to $W^{1,n}(B_1)$ since
$$
  u_i^{\ast}(a_i)=s_i=-\frac{s_i}{\cnn t_i}\log(\delta_i).
$$
\smallskip 

\textit{Step 4.} In this Step we will show that $\|\nabla v_i\|_{L^n(B_1)}\leq \|\nabla u_i\|_{L^n(\Omega)}.$ Let us denote
$$
  A_i=\int_{B_1\backslash B_{\delta_i}}|\nabla v_i|^n\quad\text{ and }
  D_i=\int_{B_{\delta_i}}|\nabla v_i|^n.
$$
A direct calculation gives that
\begin{equation}
 \label{eq:Ai estimate}
  A_i=\frac{\omega_{n-1} \,s_i^n}{\omega_{n-1}^{n/(n-1)}\,t_i^n}\int_{\delta_i}^1\frac{1}{r}dr=\frac{s_i^n}{t_i^{n-1}}.
\end{equation}
Using a change of variables and Proposition \ref{prop:Hardy littlewoood modified} (ii) gives that
$$
  D_i=\int_{B_{a_i}}|\nabla u_i^{\ast}|^n=\int_{\{u_i^{\ast}\geq s_i\}} |\nabla u_i^{\ast}|^n\leq \int_{\{u_i\geq s_i\}}|\nabla u_i|^n.
$$
Finally we get that, using \eqref{eq:proof lemma: step 2 grad norm of ui and ti. dim n}, that
$$
  \int_{B_1}|\nabla v_i|^n=D_i+A_i\leq \int_{\Omega}|\nabla u_i|^n-\int_{\{u_i<s_i\}}|\nabla u_i|^n+\frac{s_i^n}{t_i^{n-1}}\leq \int_{\Omega}|\nabla u_i|^n.
$$
\smallskip

\textit{Step 5.} In this step we show that
$$
  \lim_{i\to\infty}\frac{a_i}{\delta_i}=I_{\Omega}(0).
$$
Using the fact that $|\{u_i^{\ast}\geq s_i\}|=|\{u_i\geq s_i\}|,$
we get from \eqref{eq:proof lemma step 3:definition of ai. dim n} and the hypthesis \eqref{eq:lemma si instead of 1:ui bigger si are asymptotic balls} we get that $\rho_i-\epsilon_i\leq a_i\leq \rho_i+\epsilon_i.$
From this equation we obtain that
$$
  \frac{\rho_i}{\delta_i}\left(1-\frac{\epsilon_i}{\rho_i}\right) 
  \leq\frac{a_i}{\delta_i}\leq
  \frac{\rho_i}{\delta_i}\left(1+\frac{\epsilon_i}{\rho_i}\right).
$$
From the hypothesis \eqref{eq:lemma si instead of 1:eps over rho to zero} we know that $\epsilon_i/\rho_i\to 0.$ It is therefore sufficient to calculate the limit of $\rho_i/\delta_i$. In view the definitions of $\rho_i$,$\delta_i$ and \eqref{eq:proof:ti-lambdai goes to zero. dim n} this is indeed equal to
$$
  \lim_{i\to\infty}\frac{\rho_i}{\delta_i}=\lim_{i\to\infty}I_{\Omega}(0)e^{\cnn(t_i-\lambda_i)}=I_{\Omega}(0),
$$
which proves the statement of this step.

\smallskip

\textit{Step 6 (equality of functional limit).} Let us first show that both $u_i$ and $v_i$ converge to zero almost everywhere. For $u_i$ this holds true by hypothesis. So let $x\in B_1\backslash\{0\}$ be given and note that for all $i$ big enough
$$
  x\geq e^{-\cnn\sqrt{t_i}}\geq e^{-\cnn t_i}=\delta_i\,.
$$
Therefore we obtain from the definition of $v_i$ that
$$
  v_i(x)\leq-\frac{s_i}{\cnn t_i}\log\left(e^{-\cnn\sqrt{t_i}}\right)= \frac{s_i}{\sqrt{t_i}}\to 0,
$$
which shows the claim also for $v_i$. In view of lemma \ref{lemma:a.e. convergence and bounded by 1} it is therefore sufficient to show that 
\begin{equation}
 \label{eq:proof lemma:functional equality on ui bigger 1. dim n}
  \lim_{i\to\infty}\int_{\{u_i\geq s_i\}}\frac{e^{\alpha u_i^{n/(n-1)}}-1}{|x|^\beta}
  =I_{\Omega}^{n-\beta}(0)
  \liminf_{i\to\infty}\int_{\{v_i\geq s_i\}}\frac{e^{\alpha v_i^{n/(n-1)}}-1}{|x|^\beta}.
\end{equation}
From Proposition \ref{prop:Hardy littlewoood modified} (i) and the properties of symmetrization we get that for every $i$
$$
  \int_{\{u_i\geq s_i\}}\frac{e^{\alpha u_i^{n/(n-1)}}-1}{|x|^\beta}\leq
  \int_{\{u_i^{\ast}\geq s_i\}}\frac{e^{\alpha (u_i^{\ast})^{n/(n-1)}}-1}{(|x|^\beta)^{\ast}}
  =
  \int_{B_{a_i}}\frac{e^{\alpha (u_i^{\ast})^{n/(n-1)}}-1}{(|x|^\beta)^{\ast}}.
$$
Note that if $\beta=0$ then the inequality can actally be replaced by an equality (see Kesavan, page 14, equation (1.3.2)). For $i$ big enough $B_{a_i}\subset \Omega,$ and then $(|x|^{\beta})^{\ast}=|x|^{\beta}$ for all $x\in B_{a_i}.$ Making the substitution $x=(a_i/\delta_i)\,y$ gives
$$
  \int_{\{u_i\geq s_i\}}\frac{e^{\alpha u_i^{n/(n-1)}}-1}{|x|^\beta}\leq
  \left(\frac{a_i}{\delta_i}\right)^{n-\beta}\int_{B_{\delta_i}} \frac{e^{\alpha v_i^{n/(n-1)}}-1}{|y|^\beta}
  =\left(\frac{a_i}{\delta_i}\right)^{n-\beta}\int_{\{v_i\geq s_i\}} \frac{e^{\alpha v_i^{n/(n-1)}}-1}{|x|^\beta}.
$$
From Step 5 we therefore get that
$$
  \lim_{i\to\infty}\int_{\{u_i\geq s_i\}}\frac{e^{\alpha u_i^{n/(n-1)}}-1}{|x|^\beta}\leq
  I_{\Omega}(0)^{n-\beta}\liminf_{i\to\infty}\int_{\{v_i\geq s_i\}} \frac{e^{\alpha v_i^{n/(n-1)}}-1}{|x|^\beta},
$$ 
which proves \eqref{eq:proof lemma:functional equality on ui bigger 1. dim n}.
\smallskip

\textit{Step 7.} The last statement of the lemma ($v_i$ can concentrate only at $0$) follows from \eqref{eq:Ai estimate}, because $\lim_{i\to\infty}A_i=0$ (using that $t_i\to\infty).$ Thus $v_i$ cannot concentrate at any other point than $0.$
\end{proof}


\smallskip

We are now able to prove the main proposition of this section.
\smallskip

\begin{proof}[Proof (Proposition \ref{proposition:main domain to ball}).]
The proof is exactly the same as in the $2$-dimensional case, see details in \cite{Csato Roy Calc var Pde}: One modifies succesively the given sequence $u_i$ using Lemma \ref{lemma:main properties step 1-3 in Flucher} and  Lemma \ref{lemma:the one we sent to Struwe} and construct finally a sequence $\{v_i\}\subset W_{0,rad}^{1,n}(B_1)\cap\mathcal{B}_1(B_1),$ using Lemma \ref{main lemma si instead of 1. domain to ball},  such that
$$
  F^{\delta}_{\Omega}(0)\leq\lim_{i\to\infty}F_{\Omega}(u_i)\leq I_{\Omega}^{n-\beta}(0)\liminf_{i\to\infty}F_{B_1}(v_i).
$$
The sequence $\{v_i\}$ has to concentrate at $0,$ because it goes to zero almost everywhere. (use last statment of Lemma \ref{main lemma si instead of 1. domain to ball}, Theorem \ref{theorem:concentration alternative for singular moser trudinger} and Remark \ref{remark:F delta zero is not zero}).
\end{proof}

\smallskip

\section{Proof of the Main Result}\label{section:proof main theorem}

We now prove Theorem \ref{theorem:intro:Extremal for Singular Moser-Trudinger}. We assume that $0\in\Omega,$ the other cases are proven exactly as in the two dimensional setting. For the case $0\in \partial\Omega$ use \cite{Wang Wei} Proposition 2.4 (3), which implies that $I_{\Omega_m}(0)$ tends to zero if $\{\Omega_m\}$ is a sequence of sets whose boundary approaches the origin.
\smallskip

\begin{proof}
Let $\{u_i\}_{i\in\mathbb{N}}\subset \mathcal{B}_1(\Omega)$ be a maximizing sequence for $F_{\Omega}$. Then by Theorem \ref{theorem:concentration alternative for singular moser trudinger}, for some subsequence, either 
$$
  \lim_{i\to\infty}F_{\Omega}(u_i)=F_{\Omega}(u),
$$
or $\{u_i\}$ concentrates at some point $x\in\Omegabar.$ In the first case $u$ is an extremal function and we are done. It remains to exclude concentration. Assume, by contradiction, we have concentration at $x.$ By Proposition \ref{proposition:if u_i concentrates somewhere else than zero} we must have $x=0$ (because we have assumed $\{u_i\}$ is a maximizing sequence and $F_{\Omega}^{\sup}\neq 0$). Thus we get
\begin{equation}
 \label{eq:1}
  F_{\Omega}^{\sup}=\lim_{i\to\infty}F_{\Omega}(u_i)=F_{\Omega}^{\delta}(0).
\end{equation}
By Theorem \ref{theorem:concentration formula by domain to ball} it holds
\begin{equation}
 \label{eq:2}
  F_{\Omega}^{\delta}(0)=I_{\Omega}(0)^{n-\beta}F_{B_1}^{\delta}(0),
\end{equation}
and by Theorem \ref{theorem:supremum of FOmega on Ball}
\begin{equation}
  \label{eq:3}
   F_{B_1}^{\delta}(0)<F_{B_1}^{\sup}.
\end{equation}
Finally Theorem \ref{theorem:ball to general domain:sup inequality} states
\begin{equation}
  \label{eq:4}
   F_{B_1}^{\sup}I_{\Omega}(0)^{n-\beta}\leq F_{\Omega}^{\sup}.
\end{equation}
Combining \eqref{eq:1}-\eqref{eq:4} gives the contradiction $F_{\Omega}^{\sup}<F_{\Omega}^{\sup},$ and therefore concentration cannot occur.
\end{proof}

\section{Appendix: $n$-harmonic functions}
\label{section:n harmonic functions}

We summarize the results on $n$-harmonic functions that we have used. We only state them under the more restrictive hypothesis that are sufficient for our construction of the $n$-harmonic transplantation and we refer to references for more general versions.

\begin{definition}
Let $\Omega\subset\re^n$ be an open set and let $f\in W^{1,n}(\Omega).$ we say that $f$ is $n$-harmonic, that is 
$$
  \Delta_n f=0\quad\text{ in }\Omega,
$$
if 
$$
  \intomega|\nabla f|^{n-2}\langle \nabla f,\nabla\varphi\rangle =0\quad \text{ for all }\varphi\in C_c^{\infty}(\Omega).
$$
\end{definition}

The existence of $W^{1,n}$ solutions to the $n$-Laplace equation follows easily by direct methods of the calculus of variations. The difficult part of the next theorem is the $C^{1,\alpha}$ regularity. For the interior regularity see \cite{diBenedetto}, \cite{Tolksdorf1} or \cite{Evans C1 alpha reg}. Lieberman \cite{Lieberman} proves regularity up to the boundary under the assumption that $f$ is bounded. In this paper we only need the regularity result for very special boundary values (constants on different parts of the boundary). In this case the proof is simple, so we have included a proof of the boundedness, see Proposition \ref{proposition:hypothesis Liebermann satisfied}. For a more general version, not assuming the boundedness of $f,$ we have not found a satisfactory reference.

\begin{theorem}[Existence and Regularity]
\label{P2, Existence and regularity} There exists $0<\alpha<1,$ depending only on $n,$ with the following property.
Let $\Omega\subset\re^n$ be a bounded open smooth set, $g\in C^{1,\alpha}(\delomega)$ such that $\|g\|_{C^{1,\alpha}(\delomega)}\leq t_0$ . Then there exists a unique $f\in W^{1,n}(\Omega)$ such that 
\begin{align*}
 \Delta_n f=&0\quad\text{ in }\Omega 
 \\
 f=&g\quad\text{ on }\delomega.
\end{align*}
The solution $f$ satisfies
\begin{equation}
 \label{eq:theorem:minimum of n gradient}
  \intomega|\nabla f|^n\leq \intomega |\nabla\varphi|^n\quad\forall \varphi\in W^{1,n}(\Omega)\text{ with }\varphi=g\quad\text{ on }\delomega.
\end{equation}
Moreover if $\|f\|_{L^{\infty}}(\Omega)\leq M_0$ then $f\in C^{1,\alpha}(\Omegabar)$ and there exists a constant $C$ depending on $n,\Omega,M_0$ and $t_0$ such that 
$$
  \|f\|_{C^{1,\alpha}(\Omegabar)}\leq C(n,\Omega,M_0,t_0)
$$
\end{theorem}

A reference for the next three results is for instance \cite{Lindqvist} Theorems 2.15, 2.20 and Corollary 2.21.

\begin{theorem}[Comparison Principle]
\label{P3, Comparison Principle}
Let $\Omega\subset\re^n$ be a bounded open smooth set. Suppose $f,g\in W^{1,n}(\Omega)$ and $\Delta_n f=\Delta_n g=0$ in $\Omega.$ If
$$
  f\leq g\quad\text{ on }\delomega,
$$
then 
$$
  f\leq g\quad\text{ in }\Omega.
$$
\end{theorem}

\begin{theorem}[Harnack Inequality]
\label{P4, Harnack inequality}
Let $V$ be a compact set and $U$ open with $V\subset\subset U.$ Then there is a constant $C=C(V)$ such that 
$$
  \sup_V u \leq C \inf_V u
$$
for all $u\in W^{1,n}(U)$ such that $\Delta_nu=0$ in $U$ and $u\geq 0$ in $U.$ 
\end{theorem}

\begin{theorem}[Strong maximum principle]
\label{P1, Strong maximum principle}
Let $\Omega$ be a bounded open set with smooth boundary $\delomega$ and $u\in  W^{1,n}(\Omega).$ Suppose $\Delta_n u=0$ in $\Omega,$ $u$ is not constant and $u\geq 0.$ Then the following holds
$$ 
  \inf_{\delomega}f<f(x)<\sup_{\delomega} f\quad\text{ for all }x\in \Omega.
$$
\end{theorem}

The next theorem follows from \cite{Tolksdorf1} Theorem 1.

\begin{theorem}
\label{P5, Uniform Regularity}
Let $V$ be a compact set $V\subset\subset U,$ where is $U$ open. Then there exists $\alpha\in (0,1)$ and a  constant $C>0$ depending on $V,U,n,$ and a variable $g_0$ such that 
$$
  \|u\|_{C^{1,\alpha}(V)}\leq C\left(V,U,n,g_0\right)
$$
for all $u\in W^{1,n}(U)$ with $\Delta_nu=0$ in $U$ and $=\|u\|_{L^{\infty}(U)}\leq g_0$.
\end{theorem}

The next theorem is contained in \cite{Serrin 1} Theorem 10 and \cite{Serrin 2} Theorem 3, combined with Theorem  \ref{P2, Existence and regularity}. This is the generalization of Bocher's theorem (see \cite{Axler-Bourdon-Ramey} Theorem 3.9 page 50) to $n$-harmonic functions.

\begin{theorem}
 Let $\Omega\subset\re^n$ be a bounded open smooth set with $0\in \Omega$, $u\in W^{1,n}_{loc}(\Omega\backslash \{0\})$ with the properties
$$
  \Delta_nu=0\quad \text{ in }\Omega\backslash \{0\}.
$$
Then 
\smallskip

(a) \emph{either} $u\in L^{\infty}(\Omega),$ in which case $u$ is continuous in $\Omega.$ Moreove $u\in W^{1,n}(\Omega)$ and it solves the equation
$$
  \Delta_nu=0\quad\text{ in }\Omega,
$$
and hence $u\in C^{1,\alpha}_{loc}(\Omega)$ for some $0<\alpha<1.$
\smallskip

(b)
\emph{or} there exists a constant $k\in \re$ such that 
$$
  \intomega \langle |\nabla u|^{n-2}\nabla u,\nabla \varphi\rangle=k\varphi(0)\quad\text{ for all }\varphi\in C_c^1(\Omega).
$$
\end{theorem}

\begin{corollary}
\label{P6, Bocher Theorem}
Suppose $u\in W^{1,n}(\Omega\backslash\{0\}),$ $\Delta_nu=0$ in $\Omega\backslash\{0\}$ and $u=0$ on $\delomega.$ Then either $u$ is a constant multiple of the $n$-Green's function, i.e. $u=k G_{\Omega,0}(y)$ or $u$ vanishes identically.
\end{corollary}

The next lemma states that the hypothesis of the regularity result of  Lieberman \cite{Lieberman} is satisfied, that is, the boundedness of $f$ by $M_0$ in Proposition \ref{P2, Existence and regularity}.

\begin{proposition}
\label{proposition:hypothesis Liebermann satisfied}
Let $\Omega\subset\re^n$ be a bounded opens smooth set such that $\delomega=\bigcup_{i=1}^L \Gamma_i$ where $\Gamma_i$ are the connected smooth components of $\delomega.$ Let $k_i\in \re,$ $i=1,\ldots,L$ be constants such that $|k_i|\leq k.$  Then there exists a constant $M_0=M_0(\Omega,n,k)$ such that
$$
  \|u\|_{L^{\infty}}(\Omega)\leq M_0(\Omega,n,k)
$$
for all $u\in W^{1,n}(\Omega)$ satisfying $\Delta_n u=0$ in $\Omega$ and $u=k_i$ on $\Gamma_i$
\end{proposition}

The following lemma is required for the proof of Proposition \ref{proposition:hypothesis Liebermann satisfied} and has also been used in Lemma \ref{lemma:main properties step 1-3 in Flucher}.  

\begin{lemma}
\label{lemma1:Moser iteration}
Let $\Omega,$ $\Gamma_i$, $u$ and $k_i$ be as in Proposition \ref{proposition:hypothesis Liebermann satisfied}. Then there exists constants $C=C(\Omega,n)>0,$ $q=q(\Omega,n)>1$ such that for any $x\in \Gamma_i$ and $0<r<R$ with $B_{R}(x)\cap\delomega\subset \Gamma_i$ (that is the ball of radius $R$ does not intersect other connected parts of the boundary) one has
\begin{equation}
  \label{eq:lemma:boundary moser iteration}
\left\|(u-k_i)_+\right\|_{L^{\infty}(B_{r}(x)\cap\Omega)}\leq
  \frac{C}{(R-r)^q}\left\|(u-k_i)_+\right\|_{L^n(B_R(x)\cap\Omega)}.
\end{equation}
The same holds true for $(u-k_i)_-$ instead of $(u-k_i)_+$.
\end{lemma}

\begin{remark}
\label{remark:locality Lieberman hypothesis}(i) 
The proof of the lemma shows that $1<q<\infty$ can be choosen freely, in which case $C$ depends also on $q.$
\smallskip

(ii) The Lemma is a local result and requires $u$ to be $n$-harmonic only in a neighborhood of the ball $B_R(x)$ and constant on $\delomega\cap B_R(x).$
\end{remark}

\begin{proof}[Proof (Proposition \ref{proposition:hypothesis Liebermann satisfied}).]
By standard interior regularity, see for instance \cite{Lindqvist} Lemma 3.6, the estimate \eqref{eq:lemma:boundary moser iteration} holds true also for balls $B_R(x)$ contained in $\Omega$ (with $k_i$ replaced by $0$). Since $\delomega$ is smooth and compact and its connected components $\Gamma_i$ have positive distance between eachother, we can cover the boundary by finitely many balls $B_r(x),$ $x\in \delomega$ satisfying the estimate \eqref{eq:lemma:boundary moser iteration}. Summing up all these estimates we obtain that there exists a constant $C=C(\Omega)$ such that 
\begin{equation}
 \label{eq:u L infty estimate with u Ln}
  \|u\|_{L^{\infty}(\Omega)}\leq C \left(\|u\|_{L^n(\Omega)}+k\right).
\end{equation}
Choose a function $g\in W^{1,n}(\Omega)$ such that $g=k_i$ on $\Gamma_i$ (take for instance a bounded linear extension operator $T:W^{1-1/n,n}(\delomega)\to W^{1,n}(\Omega)$ ) so that 
$$
  \|g\|_{W^{1,n}}(\Omega)\leq C(\Omega) k.
$$
It follows from  Poincar\'e inequality ($u-g=0$ on $\delomega$) and using \eqref{eq:theorem:minimum of n gradient} that
$$
  \|u-g\|_{L^n(\Omega)}\leq C \|\nabla u-\nabla g\|_{L^n(\Omega)}
  \leq  C \left(\|\nabla u\|_{L^n(\Omega)}+\|\nabla g\|_{L^n(\Omega)}\right)
  \leq 2 C\|\nabla g\|_{L^n(\Omega)}
  \leq  2 C k
$$
From this we obtain the desired estimate for $\|u\|_{L^n}$ which we plug into \eqref{eq:u L infty estimate with u Ln} to conclude.
\end{proof}
\smallskip

We now prove Lemma \ref{lemma1:Moser iteration}, by a standard Moser iteration method.

\begin{proof}[Proof (Lemma \ref{lemma1:Moser iteration})]
Since the Lemma is a local result and $\Delta_n(u-k_i)=0,$ we can assume without loss of generality that $k_i=0$ and $u=0$ on $\Gamma_i$. The estimate \eqref{eq:lemma:boundary moser iteration} is clearly also valid for $(u-k_i)_{-}$ instead of $(u-k_i)_+$, by applying a posteriori the Lemma \ref{lemma1:Moser iteration} to $-u.$
\smallskip 

\textit{Step 1.} We shall first prove that if $\chi>1$ and $\gamma\geq n.$ Then there exists a constant $C=C(\Omega,n,\chi)>0$ such that for any $x\in \Gamma_i$ and $0<t<T$ such that $B_T(x)\cap\delomega\subset \Gamma_i$ 
\begin{equation}
  \label{eq:lemma:boundary moser iteration gamma}
\left\|u_+\right\|_{L^{\gamma\chi}(B_{t}(x)\cap\Omega)}\leq
 \left(\frac{C\gamma}{T-t}\right)^{\frac{n}{\gamma}}
  \left\|u_+\right\|_{L^{\gamma}(B_{T}(x)\cap\Omega)},
\end{equation}
provided the right hand side is finite. We fix and write henceforth for simplicity
$$
  B_T(x)=B_T\,,\quad B_t(x)=B_t\quad\text{ and }\quad \Omega_T=\Omega\cap B_T(x)\,,\quad \Omega_t=\Omega\cap B_t(x)
$$

\smallskip

\textit{Step 1.1.}
By assumption we have that 
\begin{equation}
 \label{eq:moser it hypothesis eta}
  \intomega\langle |\nabla u|^{n-2}\nabla u,\nabla \varphi\rangle=0\quad\text{ for all }\varphi\in W_0^{1,n}(\Omega).
\end{equation}
Let $u_+=\max\{0,u\}.$ Let $\xi\in C_c^1(B_T).$ We define for $m\in\mathbb{N}$ and $\beta\geq n-1$
$$
  u_m(y)=\left\{\begin{array}{rl}
              u_+(y)&\text{ if }u(y)<m 
              \\
              m&\text{ if }u(y)\geq m,
             \end{array}\right.
   \quad\text{ and }\quad \varphi=\xi^nu_m^{\beta} u\in W_0^{1,n}(\Omega),
$$
which we plug into \eqref{eq:moser it hypothesis eta}.
Deriving we obtain
$$
  \nabla\varphi=\xi^nu_m^{\beta}\nabla u+\beta\xi^n u_m^{\beta-1} u\nabla u_m
  +n\xi^{n-1}u_m^{\beta}u\nabla \xi.
$$
This gives that 
\begin{align*}
 \int_{\Omega_T}\xi^nu_m^{\beta}|\nabla u|^{n}+\beta\int_{\Omega_T}\xi^n u_m^{\beta-1} u \langle \nabla u_m,\nabla u\rangle |\nabla u|^{n-2}
 =-n\int_{\Omega_T}\xi^{n-1}u_m^{\beta} u\langle \nabla \xi,\nabla u\rangle |\nabla u|^{n-2}.
\end{align*}
In the second term $\nabla u_m$ appears, which is zero if $u\geq m,$ and therefore we can replace everywhere $u$ by $u_m$ to get
\begin{equation}
 \label{moser it before epsilon trick}
   \int_{\Omega_T}\xi^nu_m^{\beta}|\nabla u|^{n}+\beta\int_{\Omega_T}\xi^n u_m^{\beta}  |\nabla u_m|^{n}
   =-n\int_{\Omega_T}\xi^{n-1}u_m^{\beta} u\langle \nabla \xi,\nabla u\rangle |\nabla u|^{n-2}.
\end{equation}
Take $\epsilon>0,$ which we will soon chose appropriately. We use the estimate
$$
  ab\leq \frac{(n-1) a^{n/(n-1)}}{n}+\frac{b^n}{n},
$$
with $a=\xi^{n-1}|\nabla u|^{n-1} u_m^{\beta(n-1)/n}\epsilon$ and $ b=u_m^{\beta/n} u|\nabla \xi|/\epsilon$ to estimate the right side of \eqref{moser it before epsilon trick}
$$
  -n\int_{\Omega_T}\xi^{n-1}u_m^{\beta} u\langle \nabla \xi,\nabla u\rangle |\nabla u|^{n-2}
  \leq (n-1)\int_{\Omega_T}\xi^n u_m^{\beta} |\nabla u|^n\epsilon^{\frac{n}{n-1}}+\int_{\Omega_T}u_m^\beta u^n|\nabla \xi|^n\frac{1}{\epsilon^n}.
$$
So for an appropriate choice of $\epsilon=\epsilon(n)$ we obtain a constant $C=C(n)$ such that
\begin{equation}
 \label{moser it after epsilon trick}
   \int_{\Omega_T}\xi^nu_m^{\beta}|\nabla u|^{n}+\beta\int_{\Omega_T}\xi^n u_m^{\beta}  |\nabla u_m|^{n}
   \leq C\int_{\Omega_T}u_m^{\beta}u^n|\nabla \xi|^n.
\end{equation}
We now set 
$$
  w=u_m^{\frac{\beta}{n}}u,\quad\Rightarrow\quad \nabla w=\frac{\beta}{n}u_m^{\frac{\beta}{n}-1}\nabla u_m\, u+u_m^{\frac{\beta}{n}}\nabla u.
$$
Recall that if in any product $\nabla u_m$ appears, then $u=u_m$ in that product, so we obtain
\begin{align*}
  |\nabla w|^n\leq &\left(\frac{\beta}{n}\left|u_m^{\frac{\beta}{n}}\right||\nabla u_m|+u_m^{\frac{\beta}{n}}|\nabla u|\right)^n
  \leq
  2^{n-1}\left(\Big(\frac{\beta}{n}\Big)^n u_m^{\beta}|\nabla u_m|^n+u_m^{\beta}|\nabla u|^n\right)
  \smallskip \\
  &=2^{n-1}\left(\frac{\beta^{n-1}}{n^n} \beta u_m^{\beta}|\nabla u_m|^n+u_m^{\beta}|\nabla u|^n\right)
  \smallskip \\
  &\leq 2^{n-1}\left(1+\frac{\beta^{n-1}}{n^n}\right) \left( \beta u_m^{\beta}|\nabla u_m|^n+u_m^{\beta}|\nabla u|^n\right)
\end{align*}
We use this in \eqref{moser it after epsilon trick} to obtain that, renaming $2^{n-1}C$ by $C$ (as we will do henceforth)
$$
  \int_{\Omega_T}\xi^n|\nabla w|^n\leq C \left(1+\frac{\beta^{n-1}}{n^n}\right) \int_{\Omega_T} w^n|\nabla \xi|^n,
$$
Now use that
$$
  |\nabla(\xi w)|^n\leq \left(|\xi|\,|\nabla w|+|\nabla \xi|\, |w|\right)^n\leq 2^{n-1}\left(|\xi|^n|\nabla w|^n+|\nabla \xi|^n|w|^n\right).
$$
We thus get that
\begin{equation}
 \label{moser itera before Sobolev}
    \int_{\Omega_T}|\nabla(\xi w)|^n\leq C \left(1+\frac{\beta^{n-1}}{n^n}\right) \int_{\Omega_T} w^n|\nabla \xi|^n,
\end{equation}
Note that $\xi w\in W_0^{1,n}(\Omega)$ and that $W_0^{1,n}(\Omega)\hookrightarrow L^q(\Omega)$ for any $1\leq q<\infty$ by the Sobolev embedding. In particular there exists a constant $C=C(\chi,\Omega,n)$ such that
$$
  \left(\int_{\Omega_T}|\xi w|^{\chi n}\right)^{\frac{1}{\chi}}=
  \left(\int_{\Omega}|\xi w|^{\chi n}\right)^{\frac{1}{\chi}}\leq C(\chi,\Omega,n)\intomega |\nabla (\xi w)|^n
  =
  C(\chi,\Omega,n)\int_{\Omega_T} |\nabla (\xi w)|^n.
$$
So we obtain from \eqref{moser itera before Sobolev}
$$
  \left(\int_{\Omega_T}|\xi w|^{\chi n}\right)^{\frac{1}{\chi}}
  \leq 
  C \left(1+\frac{\beta^{n-1}}{n^n}\right) \int_{\Omega_T} w^n|\nabla \xi|^n
$$
We now choose $\xi$ such that $\xi=1$ in $B_t$ and
$
  |\nabla \xi|\leq 2/({T-t}).
$
We also assume that $u$ has been extende by zero outside of $\Omega$ (more precisely on the other side of $\Gamma_i$). Thus, using the definition of $w=u_m^{\beta/n} u$ we get
$$
  \left(\int_{B_t}(u_m^{\frac{\beta}{n}}u)^{\chi n}\right)^{\frac{1}{\chi}}
  \leq C
  \left(1+\frac{\beta^{n-1}}{n^n}\right)\frac{2^n}{(T-t)^n}\int_{B_T}u_m^{\beta}u^n
  \leq C
  \left(1+\frac{\beta^{n-1}}{n^n}\right)\frac{2^n}{(T-t)^n}\int_{B_T}u_+^{\beta+n}
$$
We now let $m\to\infty$ and obtain that
$$
  \left(\int_{B_t}u_+^{\chi(\beta+n)}\right)^{\frac{1}{\chi}}
  \leq C
  \left(1+\frac{\beta^{n-1}}{n^n}\right)\frac{2^n}{(T-t)^n}\int_{B_T}u_+^{\beta+n}
$$
Set now $\gamma=\beta+n$ and we get that
$$
  \left(\int_{B_t} u_+^{\gamma\chi}\right)^{\frac{1}{\chi}}\leq C\left(1+ \frac{(\gamma-n)^{n-1}}{n^n}\right) \frac{1}{(T-t)^n}\int_{B_T}u_+^{\gamma}
  \leq C \gamma^n \frac{1}{(T-t)^n}\int_{B_T}u_+^{\gamma}.
$$
From this \eqref{eq:lemma:boundary moser iteration gamma} follows.
\smallskip

\textit{Step 2.} We now use \eqref{eq:lemma:boundary moser iteration gamma} iteratively. Set 
$$
  t_j=r+\frac{R-r}{2^j},\quad \gamma_j=\chi^j n\quad\text{ for }j=0,1,2,\ldots.
$$
Note that $t_0=R,$ $t_j>t_{j+1}$ and $\gamma_0=n,$ $\gamma_{j+1}=\chi\gamma_j.$ Moreover
$t_j-t_{j+1}=(R-r)/{2^{j+1}}.$
So it follows from \eqref{eq:lemma:boundary moser iteration gamma} (by induction the right hand side of the next inequality is bounded for each $j$) that
\begin{equation}
 \label{eq:moser iter before D}
  \|u_+\|_{L^{\gamma_{j+1}}(B_{t_{j+1}})}\leq \left(\frac{C\gamma_j}{R-r}\right)^{\frac{n}{\gamma_j}}\left(2^{j+1}\right)^{\frac{n}{\gamma_j}}\|u_+\|_{L^{\gamma_j}(B_{t_j})}.
\end{equation}
Note that
\begin{align*}
  \left(\frac{C\gamma_j}{R-r}\right)^{\frac{n}{\gamma_j}}\left(2^{j+1}\right)^{\frac{n}{\gamma_j}}
  =&
  \left(\frac{C\chi^j n}{R-r}\right)^{\frac{1}{\chi^j}}\left(2^{j+1}\right)^{\frac{1}{\chi^j}}
  =
  \left(2 C \chi^jn 2^j\right)^{\frac{1}{\chi^j}}(R-r)^{-\frac{1}{\chi^j}}
  \smallskip \\
  \leq & D^{\frac{j}{\chi^j}}(R-r)^{-\frac{1}{\chi^j}},
\end{align*}
for some constant $D=D(\Omega,n,\chi).$ We now set
$a_j=\|u_+\|_{L^{\gamma_j}(B_{t_j})}$ and notice $a_0=\|u_+\|_{L^n(B_R)}.$
Thus we obtain from \eqref{eq:moser iter before D} that
$$
  a_{j+1}\leq D^{\frac{j}{\chi^j}}(R-r)^{-\frac{1}{\chi^j}}a_j.
$$
It follows by induction that 
$$
  a_{j+1}\leq D^{{\sum_{l=0}^j\frac{l}{\chi^l}}} (R-r)^{-\sum_{l=0}^{j}\frac{1}{\chi^l}}a_0\,.
$$
Notice that for all $j$ we have $\|u_+\|_{L^{n\chi^j}(B_r)}\leq a_{j+1}.$
The sums
$$
  \sum \frac{l}{\chi^l}\quad\text{ and }\quad \sum\frac{1}{\chi^l}
$$
are convergent. So letting $j\to\infty$ we obtain that 
$$
  \|u_+\|_{L^{\infty}(B_r)}\leq \frac{C}{(R-r)^q} \|u_+\|_{L^n(B_R)},
$$
which proves the lemma.
\end{proof}

\bigskip

\noindent\textbf{Acknowledgements} The first author was supported by Chilean Fondecyt Iniciaci\'on grant nr. 11150017 and  is member of the Barcelona Graduate School of 
Mathematics. He also acknowledges financial support from the Spanish Ministry of Economy and Competitiveness, through the ``Mar\'ia de Maeztu'' Programme for Units of Excellence in RD (MDM-2014-0445).
The research work of the second author is supported by the Indian "Innovation in Science Pursuit for Inspired Research (INSPIRE)" under the IVR Number: 20140000099 (IFA14-MA43).

We have obtained significant help from A. Adimurthi, who helped us understand and work out the relevant results on the $n$-Laplace equation. We would like to extend our thanks to Sandeep  for fruitful discussions on this problem. Finally, our special thanks goes to Debabrata Karmakar for  his generous help.

\end{document}